\documentclass[11pt,a4paper]{article}
\usepackage[utf8]{inputenc}
\usepackage[T1]{fontenc}
\usepackage{amsmath}
\usepackage{amsfonts}
\usepackage{amssymb}
\usepackage{graphicx}
\usepackage{hyperref}

\usepackage{algorithm}
\usepackage{algpseudocode}

\usepackage{amsfonts}
\usepackage{amssymb}
\usepackage{amsthm}
\usepackage{amsmath}
\usepackage{mathrsfs}

\usepackage{bbm}
\usepackage{enumitem}
\usepackage{gensymb}

\usepackage[left=3.00cm, right=3.00cm, top=3.00cm, bottom=2.50cm]{geometry}

\usepackage[english]{babel}
\usepackage{ulem}

\newtheorem{thm}{Theorem}[section]
\newtheorem{cor}[thm]{Corollary}
\newtheorem{lemma}[thm]{Lemma}

\theoremstyle{definition}

\newcommand{\N}{\mathbb{N}}
\newcommand{\R}{\mathbb{R}}

\newcommand{\E}{\mathbb{E}}
\newcommand{\X}{\cX}
\newcommand{\one}{\mathbbm{1}}

\newcommand{\cD}{\ensuremath{\mathcal{D}}}

\newcommand{\cN}{\ensuremath{\mathcal{N}}}

\newcommand{\cX}{\ensuremath{\mathcal{X}}}

\newcommand{\sD}{\ensuremath{\mathscr{D}}}
\newcommand{\sN}{\ensuremath{\mathscr{N}}}

\newcommand{\Dl}{\ensuremath{\Delta}}

\newcommand{\gm}{\ensuremath{\gamma}} 
\newcommand{\al}{\ensuremath{\alpha}} 
\newcommand{\bt}{\ensuremath{\beta}}  

\newcommand{\sg}{\ensuremath{\sigma}}

\newcommand{\lt}{\ensuremath{\left}}
\newcommand{\rt}{\ensuremath{\right}}

\newcommand{\charD}{\xi}

\usepackage{color}
\definecolor{darkgreen}{rgb}{0.0,0.6,0}
\newcommand{\mysout}[1]{}
\newcommand{\Change}[1]{{#1}}
\newcommand{\change}[1]{{#1}}

\title{Asymptotic limits of transient patterns in a continuous-space interacting particle system.\footnote{Keywords: interacting particle system, voter model, Fleming-Viot super-process, pattern formation, actin turnover, stepwise mutation model}}
\author{Cecilia Gonz\'alez-Tokman and Dietmar B. Oelz\\
\scriptsize{School of Mathematics and Physics,
The University of Queensland, Brisbane QLD 4072 Australia}
}

\begin{document}
\maketitle
\begin{abstract}
We study a discrete-time interacting particle system with continuous state space which is motivated by a mathematical model for turnover through branching in actin filament networks. It gives rise to transient clusters reminiscent of actin filament assemblies in the cortex of living cells.
We reformulate the process in terms of the inter-particle distances and characterise their marginal and joint distributions. We construct a recurrence relation for the associated characteristic functions and pass to the large population limit, reminiscent of the Fleming-Viot super-processes. The precise characterisation of all marginal distributions established in this work opens the way to a detailed analysis of cluster dynamics. We also obtain a recurrence relation which enables us to compute the moments of the asymptotic single particle distribution characterising the transient aggregates. Our results indicate that aggregates have a fat-tailed distribution.
\end{abstract}

\section{Introduction}
We consider an interacting particle system \cite{Liggett1999} on the real line which can be regarded as a toy model for the random turnover of actin filaments in biological cells \cite{Tam2021}. The underlying modelling assumption is that through ARP-2/3 dependent branching~\cite{Gautreau2022} new actin filaments are positioned into the vicinity of existing actin filaments and that for every newly branched filament another one disassociates (Figure~\ref{fig_actin_filaments}A). 

\begin{figure}
	\centering
	\includegraphics[width=0.9\linewidth]{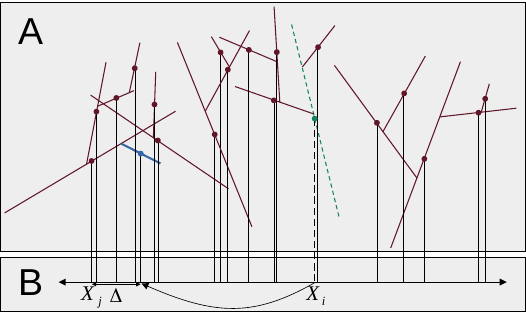}
	\caption{A: Sketch of an actin filament network in which filaments emerge by branching at -- approximately -- the characteristic $70\degree$ angle from existing actin filaments~\cite{Gautreau2022}. B: The projection of filament centre points onto a coordinate axis can be roughly described by a point process by which the simultaneous rapid disassociation of one filament (green-dashed) and branching of another filament (blue-bold) is interpreted as repositioning particle $i$ to a position at distance $\Delta$ from particle $j$.} 
 \label{fig_actin_filaments}
\end{figure}

In the context of the toy model, we idealise actin filaments as point objects, \Change{e.g. by projecting filament centre points onto a coordinate axis (Figure~\ref{fig_actin_filaments}B)}. \Change{We} ignore any dynamics other than disassociation and branching events occurring at regular intervals of time \Change{and model the spatial distance between centre points of nucleated filaments and parental filaments along the axis of projection by a given symmetric distribution of distances}. This leads us to consider the following discrete time Markov process with state space $\R^N$ for the positions of \Change{$N\geq2$} interacting particles, $X_k^n \in \R$, $k \in \{1, \ldots, N\}$ \Change{at time $n\in \mathbb N_0$. The initial positions $X_k^0$ are independent and identically distributed.
Assuming the process has been defined up to time $n$, at time $n+1$, an ordered pair of distinct indices $i_{n+1}, j_{n+1} \in \{1, \dots, N\}$ is selected with uniform probability on the set of allowed pairs, and independently of all previous choices. Also, a random offset $\Delta_{n+1}\in \mathbb R$ is selected, independently of all previous choices, and distributed according to a fixed (time-independent), symmetric\footnote{Symmetry of this distribution is assumed to simplify the presentation, but the arguments below could also be adapted to asymmetric distributions.} centred distribution with finite variance $\sigma^2>0$. $X_k^{n+1} \in \R$ is defined as follows:}
\begin{equation}
\label{equ_process1}
	X_k^{n+1}=
	\begin{cases}
		X_k^n  & k \neq i_{n+1} \; , \\
		X_{j_{n+1}}^n + \Delta_{n+1}  & k = i_{n+1} \;,
	\end{cases}
\end{equation}	
\mysout{where discrete time is denoted by $n=1,2, \ldots$ and the initial positions $X_k^0$ are identically distributed.} The process models the jump of particle $i_{n+1}$ into the vicinity of particle $j_{n+1}$. \mysout{where both indices are sampled from the uniform distribution on the set $\{1, ..., N\}$ under the constraint that $i_{n+1} \neq j_{n+1}$} The new particle position deviates from the position of particle $j_{n+1}$ by the random offset $\Delta_{n+1}$.
\mysout{ which are assumed to be iid according to a given centred uni-variate distribution with variance $\sg^2$.}
This can be interpreted as the simultaneous disassociation of one particle and creation of another particle.

\begin{figure}
	\centering
	\includegraphics[width=0.9\linewidth]{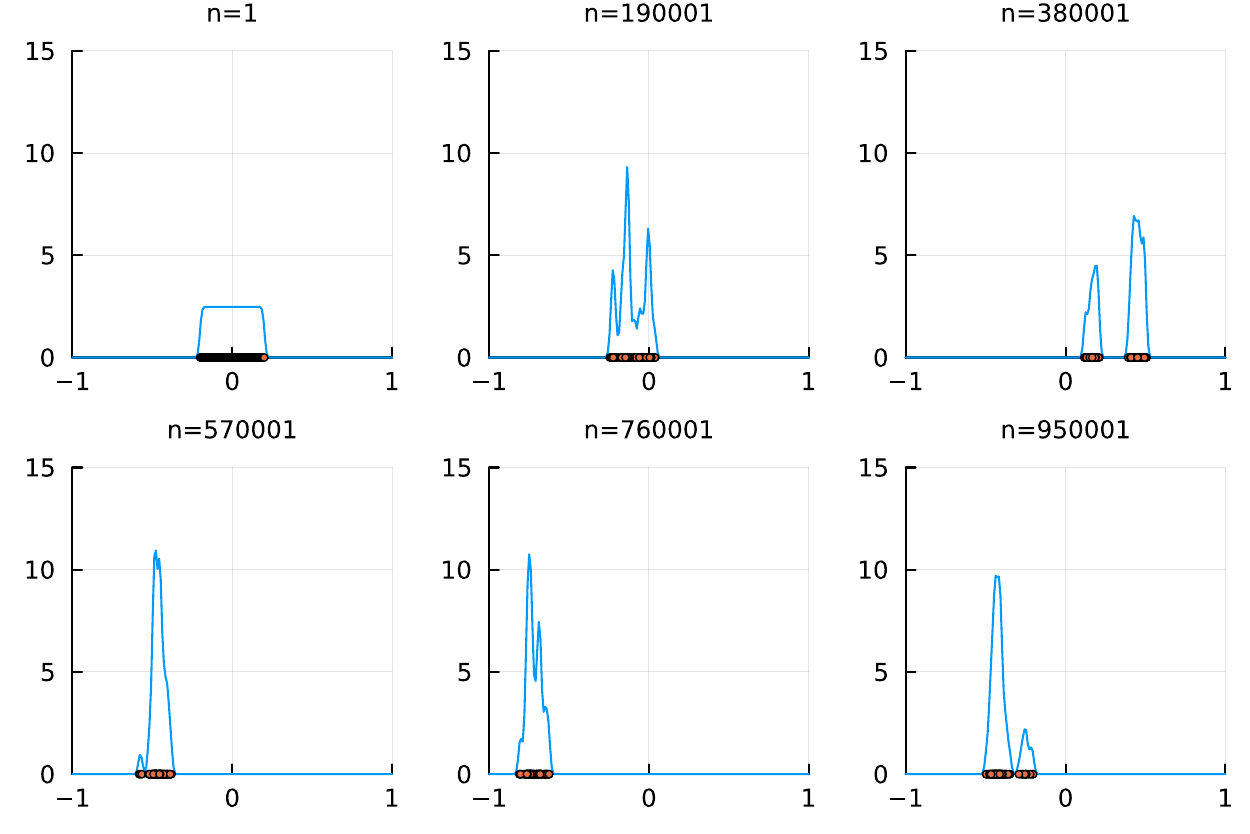}
	\caption{Stochastic simulation of particle positions $X_1^n$,\dots,$X_{100}^n$ for $\sigma=0.1$: Scatter plot and kernel density estimator (Gaussian kernel with variance $0.01^2$).} 
 \label{fig:sequenceofframes}
\end{figure}

Pattern formation is the most intriguing feature of the solutions to this process, namely the emergence of transient aggregates (see Figure~\ref{fig:sequenceofframes}). In the context of the cellular actin actin filament network, it can be speculated that these aggregates may stabilise higher order cytoskeleton assemblies such as stress fibres \cite{Lappalainen2012}, or even favour their spontaneous assembly, for example of ventral stress fibres \cite{Lappalainen2021}. \mysout{Note that other mechanisms involved in patterning the actin cytoskeleton involve mechanical forces exerted through molecular motor proteins.} 

\Change{Filament turnover is indeed a characteristic feature of cytoskeletal networks~\cite{Salbreux2016}. Patterns are regarded to be a combined effect of turnover and dynamics due to mechanical forces~\cite{SALBREUX2012} exerted actively by motor proteins and passively by cross-linker proteins~\cite{THORESEN2011}. The biased turnover represented by the toy model discussed in this study can be interpreted as a coarse grained description of local mechanisms of turnover which involve severing and branching as observed in keratocyte fragments~\cite{Gautreau2022,RAZBENAROUSH2017}. Note that local turnover typically favours the emergence of aggregates in contrast to spatially uniform turnover which counteracts pattern formation~\cite{Tam2021}.}

\Change{Broadly speaking the process can be interpreted as birth and death process by which in every step $n \rightarrow n+1$ the particle $X_{i_{n+1}}$ is removed and replaced by a new particle at $X_{j_{n+1}} + \Delta_{n+1}$.}
Alternatively, the component $X_k^n \in \mathbb{R}$ can be interpreted as the opinion of individual $k$ at time $n$ within the entire bandwidth of potential opinions represented by the set of real numbers. As such, this model is reminiscent of the classical voter model \cite{Liggett1999} and describes the random dynamic of a collective set of opinions $X_k^n$ for $k=1,...,N$. 
\Change{Furthermore it also has features of the classical Moran process~\cite{moran_1958}. As such it could be interpreted as replication-mutation dynamics of individuals which are each characterised by a single allele from the set of available alleles represented by $\mathbb{R}$. In this context the random variable $\Delta_{n+1}$ represents the mutation of the allele of individual $X^{n}_{i_{n+1}}$.}
\Change{This is reminiscent of the step-wise mutation model~\cite{DawsonHochberg1982}  as formulated in \cite{FlemingViot1979} where the notation of the classical Ohta-Kimura step-wise mutation model in population genetics \cite{OhtaKimura1973} was extended to non-discrete sets of possible types of individuals.}
\Change{We remark that the process we consider involves stochastic resetting of inter-particle distances to $\Delta_{n+1}$; \Change{we refer the reader to \cite{NagarGupta_2023} for a recent review on this subject.} We therefore expect that this process exhibits non-equilibrium steady states (NESS) \cite{EvansMajumdarPRL2011}, although we do not pursue this direction in this work.}

In this study we are interested in investigating cluster formation and specifically in describing the stationary distribution of relative positions \Change{and opinions, respectively, as well as} inter-particle distances.
We remark that an understanding of the joint distribution of inter-particle distances translates into direct information about cluster formation, as a cluster can be characterised by small inter-particle distances for elements within the cluster, and large inter-particle distance between elements in different clusters (or outside the clusters).

We consider the case of finitely many ($N$) particles, as well as the limit as $N\to\infty$.
Our approach is similar to Kingman's analysis \cite{Kingman76} of the step-wise mutation model, focusing on the particles' relative configurations -- without centering, the process has no stationary probability distribution -- and relying on characteristic functions.
Kingman noticed the crucial significance of the limiting measure as $N\to \infty$, and the fact that its properties  ``\textit{are unfortunately not easy to calculate''.}
Such a limit is related to Fleming-Viot superprocesses, 
a topic which has received significant attention in the literature, see e.g. \cite{Etheridge,bookDawsonGreven2014,Ferrari2007,BiswasEtheridgeKlimek2021,ChampagnatHass22} and reference therein.

We start by exploring the process \eqref{equ_process1} numerically and reformulating it in terms of inter-particle distances in section~\ref{particle distances}. Note that through this reformulation the process \eqref{equ_process1} is naturally symmetrised, which facilitates the analysis significantly. In section~\ref{S:1diff} we compute the time-asymptotic distribution of inter-particle distances, defined in \eqref{eq:diffs}, where we make use of the characteristic functions (Fourier transforms). \mysout{ which is the second major tool employed in our analysis.}

The \Change{centred} particle positions can then be recovered from the inter-particle distances, through \eqref{equ_Xreconstruct}.\Change{\footnote{Note that the center of mass of the original process \eqref{equ_process1} cannot be recovered from inter-particle distances.}}
\Change{To illustrate our approach, }
in section~\ref{sec_simple}  we consider small particle ensembles of size $N=2$ and $N=3$ and directly compute their time-asymptotic marginal distributions, making use of \mysout{the} our previous results on the distribution of inter-particle differences. 
In section~\ref{sec_N} we characterise the time-asymptotic marginal distribution of a single particle in the process \eqref{equ_process1} relative to the center of mass,
as well as the time-asymptotic joint distribution of inter-particle distances, 
based on recurrence relations for the associated characteristic functions. 
Finally, in section~\ref{equ_Nlimit} we pass to the limit of large population size ($N\to\infty$). 
The results of this section provide a way to compute the moments of the asymptotic particle distribution through a recurrence relation. In principle it allows one to compute all the moments of the large-population asymptotic single-particle distribution, although for practical reasons the numerical evaluation is only feasible for moments of order $\leq 60$ approximately.

\section{Preliminaries}
\label{particle distances}

For \Change{$N\geq 2$},
let us consider the renormalised $N$-particle process 
\begin{equation*}
\bar X_k^n=\frac{X_k^n-\langle X^n \rangle}{\sqrt{N}} \; ,
\end{equation*}
where $n\in \N$, $1 \leq k \leq N$, $X_k^n$ is as in \eqref{equ_process1} and $\langle X^n \rangle=\frac{1}{N} \sum_{k=1}^N X_k^n$ is the mean position at time $n$. Note that its time evolution is given by $\langle X^{n+1} \rangle=\langle X^n \rangle+\frac{1}{N} (X^n_{j_{n+1}}+ \Delta_{n+1}-X^n_{i_{n+1}})$. The renormalised process \Change{$\bar \X_N^n=(\bar X_1^n, \dots, \bar X_N^n)$} therefore satisfies
\begin{equation}
\label{eq:process}
\bar X_k^{n+1}=\begin{cases}
	\bar X_k^n + \frac{1}{N} \left( \bar X_{i_{n+1}}^n - \bar X_{j_{n+1}}^n - \frac{\Delta_{n+1}}{\sqrt{N}}  \right) & k \neq i_{n+1}  \; ,\\
	\bar X_{j_{n+1}}^n + \frac{\Delta_{n+1}}{\sqrt{N}}  + \frac{1}{N} \left( \bar X_{i_{n+1}}^n - \bar X_{j_{n+1}}^n - \frac{\Delta_{n+1}}{\sqrt{N}}  \right) & k = i_{n+1}  \; .
\end{cases}
\end{equation}	
Note that the renormalised process satisfies 
 \begin{equation}
 	\label{equ_sum}
	\bar X_1^n+\bar X_2^n+ \ldots + \bar X_N^n=0,
\end{equation}	
for all $n=1,2,...$. Taking the square of this identity and the expectation on both sides  \Change{yields
\begin{equation*}
	\Change{\sum_{k=1}^N\sum_{\substack{l=1\\ l\neq k}}^{N}\mathbb{E}[\bar X_k^n \, \bar X_l^n]}=- \sum_{k=1}^N\mathbb{E}[(\bar X_k^n)^2]. 
\end{equation*}	
We note that, by induction, the process $\bar \X_N^n$ is exchangeable, i.e. symmetric with respect to permutation of particle labels. More precisely, if $(k_1, \dots, k_N)$ is a permutation of $(1,...,N)$, then the joint distribution of $(X^n_1, \dots, X^n_N)$ coincides with the joint distribution of  $(X^n_{k_1}, \dots, X^n_{k_N})$, for every $n\geq0$. This property gives rise to similar symmetries for $\bar \X_N^n=(\bar X^n_1, \dots, \bar X^n_N)$.
Thus, we conclude that for every $1\leq k, l \leq N$ with $k\neq l$,}

\begin{equation*}
	\mathbb{E}[\bar X_k^n \, \bar X_l^n]=- \frac{1}{(N-1)} \mathbb{E}[(\bar X_k^n)^2]. 
\end{equation*}	
In particular, the single particle positions $\bar X_k^n$ are not independent. 

As an alternative to working with the symmetrised process \eqref{eq:process} we introduce the process of inter-particle distances $D_{lk}^n=\bar X_l^n-\bar X_k^n$. They satisfy
\begin{equation}
\label{eq:diffs}
D_{lk}^{n+1}=\begin{cases}
	D_{lk}^{n} &  \Change{i_{n+1}\notin \{ k, l\} }  \; ,
	\\
	D_{j_{\change{n+1}}k}^{n}+ \frac{\Delta_{n+1}}{\sqrt{N}}   & l = i_{n+1}, k \neq j_{n+1}  \; ,
	\\
	D_{lj_{\change{n+1}}}^{n} - \frac{\Delta_{n+1}}{\sqrt{N}}   & k = i_{n+1}, l \neq j_{n+1} \; , 
	\\
	\frac{\Delta_{n+1}}{\sqrt{N}}   & l = i_{n+1}, k=j_{n+1} \; , 
	\\
	- \frac{\Delta_{n+1}}{\sqrt{N}}   & k = i_{n+1}, l=j_{n+1}  \; ,
\end{cases}
\end{equation}
which has a significantly simpler structure than \eqref{eq:process}. 
Note that the original process $\bar X_k^n$ can be recovered from the process of inter-particle distances. Indeed,
 identity \eqref{equ_sum} implies that $N \bar X_1^n=(\bar X_1^n-\bar X_2^n)+...+(\bar X_1^n-\bar X_N^n)$ implying that 
\begin{equation}
	\label{equ_Xreconstruct}
\bar X_1^n = \frac1N\sum_{j=2}^N D_{1j}^n \; ,
\end{equation}
and similar identities hold for $\bar X^n_k$, $k > 1$ \cite{asselah_ferrari_groisman_2011}. 

Our strategy in this paper is to work with the process $D_{ij}^n$, namely with the characteristic functions characterising its marginal distributions, and to finally reconstruct the law of single particles using \eqref{equ_Xreconstruct}.
\Change{We will also rely on the symmetries of the process $(D_{12}^n, \dots, D_{1N-1}^n)$ under permutation of indices,  which follow directly from the symmetries of $\bar \X_N^n$.}

\Change{Our arguments will repeatedly rely on L\'evy's continuity theorem, ensuring that convergence in distribution of the random variables (or random vectors) $Y_1, Y_2, \dots \in \mathbb R^k$, with laws $\mu_1, \mu_2, \dots$, to a random variable (or random vector) $Y_\infty \in \mathbb R^k$, with law $\mu_\infty$, is equivalent to pointwise convergence of their characteristic functions. 
More precisely, for every $n=1,2,\dots,\infty$, let $\phi_n(t):=\mathbb E(e^{it\cdot Y_n})$, $t\in\mathbb R^k$, be the characteristic function of $Y_n$. Then $Y_1, Y_2, \dots$ converge in distribution to $Y_\infty$ -- or equivalently, $\mu_1,\mu_2,\dots$ converge weakly to $\mu_\infty$ --
if and only if for every $t\in\mathbb R^k, \lim_{n\to\infty}\phi_n(t)=\phi_\infty(t)$.
See e.g. \cite[\S 26 and \S29]{Billingsley}.}

\paragraph{Notation.}
Throughout this work,  $\lambda_N$ and $\xi_N$ will denote the law and characteristic function of $\Change{\Delta_n}/\sqrt{N}$, and $\lambda=\lambda_1$,  $\xi=\xi_1$.
Also, 
$\nu_{N,n}$ and $\phi_{N,n}$ will denote the law and characteristic function of normalised particle positions, respectively -- e.g. $\bar X_1^{n}$. 
Similarly, $\mu_{N,n}$ and $\psi_{N,n}$ will be the law and characteristic function of inter-particle distances -- e.g. $D_{12}^{n}$. 
For $1\leq k < N$, $\mu_{N,n}^k$ and $\psi_{N,n}^k$ will be the law and characteristic function of $k$ inter-particle distances -- e.g. $(D_{12}^{n}, D_{13}^{n}, \dots, D_{1k+1}^{n})$.

\section{Asymptotic distribution of inter-particle distances}\label{S:1diff}
In this section we characterise the asymptotic distribution of inter-particle distances\mysout{ $\mu_N$}.
Exploiting the symmetry of the process with respect to permutation of particle labels, we analyse 
$D_{12}^{n} = \bar X_1^{n}-\bar X_2^{n}$, without loss of generality.

\begin{thm}[Marginal distribution of inter-particle distances]\label{lem:diffs} 
\Change{Let $N\geq2$}. Let $\bar \X_N^n=(\bar X_1^n, \dots, \bar X_N^n)$ be as in \eqref{eq:process} and  $D_{12}^{n} = \bar X_1^{n}-\bar X_2^{n}$. Let $\mu_{N,n}$ be the law of  $D_{12}^{n}$.
\Change{Then $\mu_{N,n}$ converges weakly as $n\to\infty$ to a measure $\mu_N$, whose}
characteristic function  is 
 \begin{equation}\label{eq:CFDiffsFormula}
 \psi_N(s)=\frac{ \charD(s/\sqrt{N})}{N-1 - (N-2) \charD(s/\sqrt{N})} \Change{, \quad s\in \mathbb R}.
 \end{equation}
\Change{In other words, $D_{12}^n$ converges in distribution to a random variable with characteristic function given by \eqref{eq:CFDiffsFormula}. Furthermore, if $N>2$},
\begin{equation}
\label{eq:nuDiffsFormula}
\mu_N =\frac{1}{N-2} \sum_{k=1}^\infty \lt(\frac{N-2}{N-1}\rt)^k  (\lambda_N)^{\ast k},
\end{equation}
where $\nu^{\ast k}:=\underbrace{\nu \ast \nu \ast\dots\ast \nu}_{\text{$k$ times}}$
\Change{is the $k$-fold convolution of $\nu$ with itself.}
 \end{thm}
\begin{proof}
Let  $D_{12}^{n}$ and $\mu_{N,n}$ be as in the statement.
It follows from \eqref{eq:diffs} that for every $n\in\N$,
\begin{equation}
	\label{eq:diffs12}
	D_{12}^{n+1}=\begin{cases}
	D_{j_{\change{n+1}}2}^{n}+\frac{\Delta_{n+1}}{\sqrt{N}}   &  1=i_{n+1}, 2 \neq j_{n+1}  \; , \\
		D_{1j_{\change{n+1}}}^{n} -\frac{\Delta_{n+1}}{\sqrt{N}}   &  2=i_{n+1}, 1 \neq j_{n+1} \; , \\
		\frac{\Delta_{n+1}}{\sqrt{N}}   & 1= i_{n+1}, 2=j_{n+1} \; , \\
		-\frac{\Delta_{n+1}}{\sqrt{N}}   & 2 = i_{n+1}, 1=j_{n+1}  \; ,\\
		D_{12}^{n} &  i_{n+1} \notin \{1,2\}  \;  . 
		\end{cases}
\end{equation}
Note that the indices $(i_{n+1}, j_{n+1})$ are uniformly distributed on a set of cardinality $N(N-1)$. For that reason, writing
\begin{equation}
\label{equ_abg}	
\al=\frac{2(N-2)}{N(N-1)} \;, \quad \bt= \frac{2}{N(N-1)} \quad \text{and} \quad \gm=\frac{N-2}{N} \; ,
\end{equation}
where $\alpha+\beta+\gm=1$, the first two cases in \eqref{eq:diffs12} occur each with probability $\alpha/2$, cases $3$ and $4$ occur each with probability $\beta/2$ and the last case occurs with probability $\gm$. 

Since $D_{kl}^{n}$ and $\Delta_{n+1}$ are independent for every $1\leq k,l \leq N$, and $\Dl_{n+1}$ is symmetric we have that for every $n\in\N$ the laws of $D^n_{12}$ at subsequent timesteps are coupled through
 \begin{equation}\label{eq:RecNun}
 \mu_{N,n+1}=\alpha \, \mu_{N,n} \ast \lambda_N + \beta \, \lambda_N + \gm \, \mu_{N,n} \; ,
 \end{equation}
where we recall that $\lambda_N$ is the law of $\Delta_{n+1}/\sqrt{N}$.
Also recall that $\psi_{N,n}(s) = \E(\exp(is D_{12}^{n}))$ is the characteristic function of $D_{12}^{n}$, i.e. the Fourier transform of $\mu_{N,n}$, and $\charD_{N}(s)=\charD(s/\sqrt{N})$ is the characteristic function of $\frac{\Delta}{\sqrt{N}}$. Fourier transformation of  \eqref{eq:RecNun} implies that 
$
\psi_{N,n+1}= (\al \charD_{N} +\gm) \psi_{N,n} + \bt \charD_{N}
$. For every $s \in \mathbb{R}$ this recurrence relation is a contraction since $| \psi_{N,n+1}(s)-\psi_{N,n}(s) | \leq C | \psi_{N,n}(s)-\psi_{N,n-1}(s) |$ where $C=\Change{\sup_{s\in\mathbb R}|\al \charD_{N}(s) +\gm|} \leq \al +\gm <1$ due to uniform boundedness of the characteristic function,
\Change{i.e. $\sup_{s\in\mathbb R} |\charD_{N}(s)|=1$}. 
This guarantees pointwise convergence of $\psi_{N, n}$ as $n \rightarrow \infty$ to the fixed point satisfying
\begin{equation*}
	\psi_N= (\al \charD_{N} +\gm) \psi_N + \bt \charD_{N} \; ,
\end{equation*}
which is equivalent to \eqref{eq:CFDiffsFormula} . 

Since $\charD_{N}$ is continuous, $\psi_N$ is continuous at $0$ and L\'evy's continuity theorem ensures that 
$\mu_{N,n}$ converges weakly, say to $\mu_N$, with characteristic function $\psi_N$.

We now show that \eqref{eq:nuDiffsFormula} holds \Change{when $N>2$}. Note that in this case \eqref{eq:CFDiffsFormula} can be written in terms of a geometric series, 
$\psi_N=\frac{1}{N-2} \sum_{k=1}^\infty \left(\frac{N-2}{N-1} \charD_{N} \right)^k$. The inverse Fourier transform of this identity implies \eqref{eq:nuDiffsFormula}. 
\end{proof}

As an application of Theorem~\ref{lem:diffs} in a context where the probability measures involved have a pdf, we have the following.
\begin{cor}\label{cor:normal}
In the setting of Theorem~\ref{lem:diffs}, \Change{if $N>2$ and}
$\Dl_n \sim \cN(0,\sg^2)$, then
$\mu_N$ is a probability measure with pdf
\[
\frac{1}{N-2} \sum_{k=1}^\infty \lt(\frac{N-2}{N-1}\rt)^k \frac{1}{\sg\sqrt{2\pi k/N}}  e^{\frac{-y^2}{2k\sg^2/N}}.
\]
\end{cor}
\begin{proof}
The statement follows directly from the fact that, in this case,
 $(\lambda_N)^{\ast k}$ is the law of a normal random variable with mean $0$ and variance $k\sg^2/N$.
\end{proof}

We can also identify the limit of the inter-particle distance distribution, in the infinite particle limit. This limit is universal, in that it does not depend on the distribution of $\Delta_n$ other than through its variance.  We include a proof which will be referred to in the sequel, but see e.g. \cite[Proposition 2.2.9]{KKP-Laplace} for a similar statement. This result is illustrated in Figure~\ref{fig:empiricaldistribdifferences}.

\begin{cor}[Limit distribution of inter-particle distances]\label{cor:Laplace}
In the setting of Theorem~\ref{lem:diffs},
the limit $\mu_\infty:=\lim_{N\to \infty} \mu_N$ exists, and it is the law of a Laplace distribution with mean 0 and 
variance $\sg^2$ (scale parameter $b=\tfrac{\sg}{\sqrt{2}}$). Hence,  
the characteristic function of $\mu_\infty$ is $\frac{2}{2+s^2\sg^2}$, and its 
 pdf  is 
$
\frac{1}{\sqrt{2}\sg}e^{-\sqrt{2}|y|/\sg}.$
\end{cor}
\begin{proof}
Let $\charD_{N}$ the characteristic function of $\frac{\Delta_n}{\sqrt{N}}$.
Using properties of characteristic functions,  $\charD_{N}(s) =1- \frac{\sg^2s^2}{2N} + o(\frac{s^2}{N})$, as $N\to\infty$.
Substituting this into \eqref{eq:CFDiffsFormula} yields\Change{, for every $s\in \mathbb R$,}
\begin{equation}\label{eq:phiNlim}
	\psi_N(s)=\frac{1- \frac{\sg^2s^2}{2N} + o(\frac{s^2}{N})}{(N-1)- (N-2)(1- \frac{\sg^2s^2}{2N} + o(\frac{s^2}{N}))}=\frac{1- \frac{\sg^2s^2}{2N} + o(\frac{s^2}{N})}{1+ (N-2)(\frac{\sg^2s^2}{2N} + o(\frac{s^2}{N}))}.
\end{equation}

Hence,  $\psi_N(s) = \frac{2}{2+s^2\sg^2} + \mathcal O (\frac{s^2}{N})$ as $N\to\infty$. 
The statement follows from the fact that 
$\frac{2}{2+s^2\sg^2}$ is 
the characteristic function of a Laplace random variable with density $\exp(-|x|/b)/(2 b)$ and scale parameter $b=\sqrt{\sigma^2/2}$. 
\end{proof}

\begin{figure}
	\centering
	\includegraphics[width=0.5\linewidth]{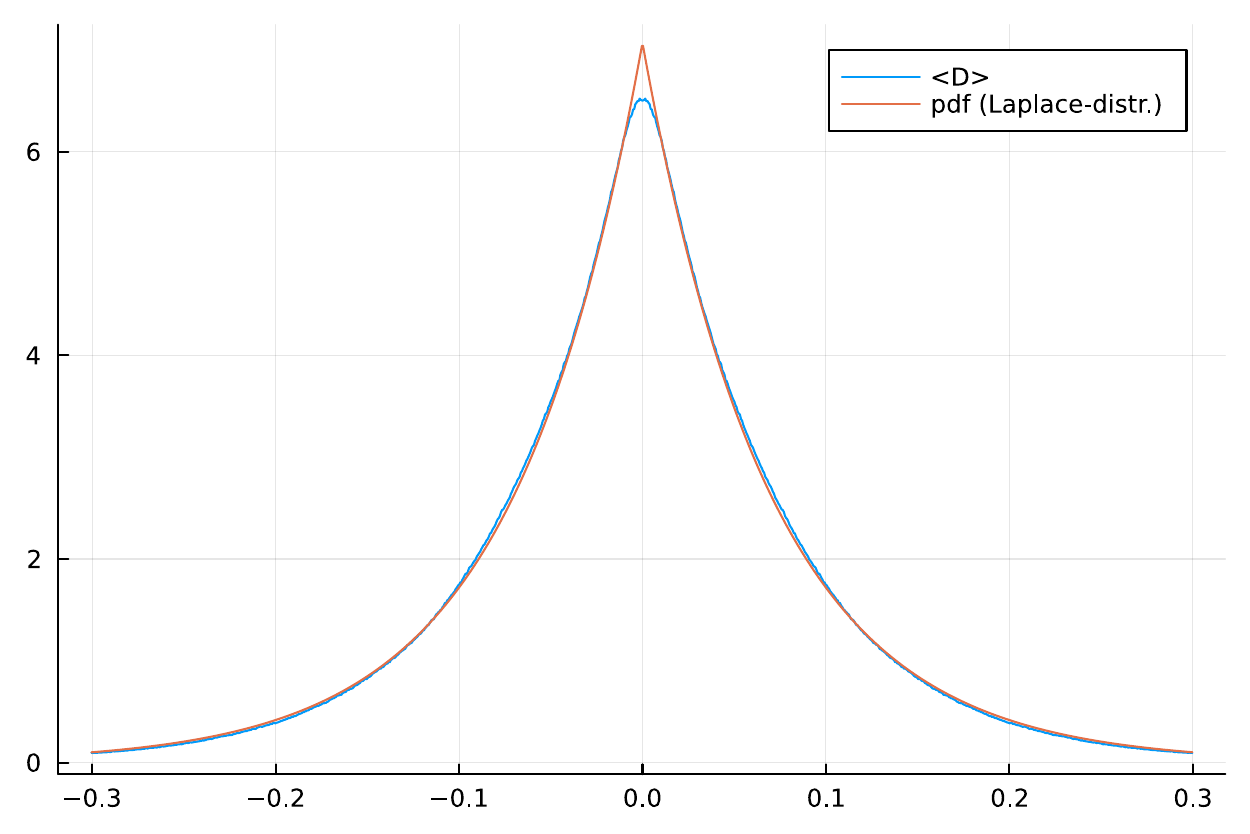}
	\caption{Time-averaged empirical marginal distribution of inter-particle distance $D_{ij}$ with $N=100$, $\sigma=0.1$ \change{and $\Delta_{n+1} \sim N(0, \sigma^2)$}, spatially renormalised by $1/\sqrt{N}$ (blue) vs Laplace distribution with scale parameter $b=\sqrt{\sigma^2/2}$ (red).}
	\label{fig:empiricaldistribdifferences}
\end{figure}

\section{Small particle ensembles}
\label{sec_simple}
In this section, we investigate the asymptotic distribution of the process $(\bar \X_N^n)_{n\in \N}$ for $N=2, 3$. 
Recall that $\lambda_N$ is the  law of $\frac{\Delta_n}{\sqrt{N}}$, and $\charD$ the characteristic function of $\Dl_n$.
\mysout{While the analysis of this section can easily be adapted to the case of asymmetric $\Dl_n$, for simplicity,} 
Throughout this section we will assume that  
$\lambda_N$ has a pdf denoted by $\rho_N^\Dl$.

\subsection{The simple case of two particles}\label{S:2part}
When $N=2$, the evolution of the first particle in the process \eqref{eq:process} reduces to 
\begin{equation*}
	\bar X_1^{n+1}=\begin{cases}
		 \frac{1}{2}  \frac{\Delta_{n+1}}{\sqrt{2}}   &  i_{n+1} = 1, \  j_{n+1}=2 \; , \\	
		 -\frac{1}{2}  \frac{\Delta_{n+1}}{\sqrt{2}}  & i_{n+1} = 2,    \ j_{n+1} = 1 \; .
	\end{cases}
\end{equation*}

 Thus, the 
  characteristic function of  $\bar X_1^{n+1}$ is $\phi_2(s)=\charD \bigl(\tfrac{s}{2\sqrt{2}}\bigr)$, independently of $n$.
Since $\bar X_2^{n}=- \bar X_1^{n}$ for every $n\in \N$, the asymptotic distribution of the process  $(\bar \X_2^n)_{n\in 
\N}$ is characterised by $\phi_2$. 
\Change{ Indeed, if $\phi_{2,n}^2: \mathbb R^2 \to \mathbb C$ is the characteristic function of $\bar \X_2^n$, then $\phi_{2,n}^2(s_1, s_2)= \mathbb E(e^{i(s_1\bar X_1^n + s_2 \bar X_2^n)})=\mathbb E(e^{i\bar X_1^n (s_1- s_2)})=\phi_2(s_1-s_2)$.}

\subsection{The case of three particles}
When $N=3$, the evolution of the first particle  in the process \eqref{eq:process} is given by

\begin{equation*}
	\label{eq:process3}
	\bar X_1^{n+1}=
	\begin{cases}
		 \frac{1}{3} \bar X_{1}^n +\frac{2}{3}\bar X_{2}^n + \frac{2}{3} \frac{\Delta_{n+1}}{\sqrt{3}}     & i_{n+1}=1, j_{n+1}=2 \; , \\
		\frac{1}{3} \bar X_{1}^n +\frac{2}{3}\bar X_{3}^n + \frac{2}{3} \frac{\Delta_{n+1}}{\sqrt{3}}  & i_{n+1}=1, j_{n+1}=3  \; ,\\
		\frac{2}{3} \bar X_{1}^n +\frac{1}{3}\bar X_{2}^n - \frac{1}{3} \frac{\Delta_{n+1}}{\sqrt{3}} & i_{n+1}=2, j_{n+1}=1   \; ,\\
		\frac{2}{3} \bar X_{1}^n +\frac{1}{3}\bar X_{3}^n - \frac{1}{3} \frac{\Delta_{n+1}}{\sqrt{3}}  &  i_{n+1}=3, j_{n+1}=1\; ,\\
		 \bar X_{1}^n+\frac{1}{3} \bar X_{2}^n -\frac{1}{3}\bar X_{3}^n - \frac{1}{3} \frac{\Delta_{n+1}}{\sqrt{3}}  & i_{n+1}=2, j_{n+1}=3\; ,\\
		\bar X_{1}^n-\frac{1}{3} \bar X_{2}^n +\frac{1}{3}\bar X_{3}^n - \frac{1}{3} \frac{\Delta_{n+1}}{\sqrt{3}}  & i_{n+1}=3, j_{n+1}=2 \; .
	\end{cases}
\end{equation*}	
Recalling that $\bar X_1^{n}+ \bar X_2^{n}+\bar X_3^{n}=0$ for all $n\in\N$ this can be symmetrised as follows,
\begin{equation*}
	\bar X_1^{n+1}=\begin{cases}
		 \frac{1}{3} \lt(\bar X_2^{n} - \bar X_3^{n} + 2\frac{\Delta_{n+1}}{\sqrt{3}} \rt)   & i_{n+1} = 1, \  j_{n+1}=2 \; , \\
		  \frac{1}{3} \lt(\bar X_3^{n} - \bar X_2^{n} + 2\frac{\Delta_{n+1}}{\sqrt{3}} \rt)   & i_{n+1} = 1, \  j_{n+1}=3 \; , \\
		   \frac{1}{3} \lt(\bar X_1^{n} - \bar X_3^{n} - \frac{\Delta_{n+1}}{\sqrt{3}} \rt)   & i_{n+1} = 2, \  j_{n+1}=1\; ,  \\
		 \frac{1}{3} \lt(\bar X_1^{n} - \bar X_2^{n} - \frac{\Delta_{n+1}}{\sqrt{3}} \rt)   & i_{n+1} = 3, \  j_{n+1}=1 \; , \\
	\frac{1}{3} \lt(2(\bar X_1^{n} - \bar X_3^{n}) - \frac{\Delta_{n+1}}{\sqrt{3}} \rt)   & i_{n+1} = 2, \  j_{n+1}=3 \; , \\
		  \frac{1}{3} \lt(2(\bar X_1^{n} - \bar X_2^{n}) - \frac{\Delta_{n+1}}{\sqrt{3}} \rt)   & i_{n+1} = 3, \  j_{n+1}=2 \; .
	\end{cases}
\end{equation*}
This implies that the particle positions can be computed in terms of the inter-particle distances.
To illustrate how this is done, we focus our attention in the case where the distributions involved have densities\footnote{In \S\ref{sec_N}, we present a general argument valid for all probability measures.}.

\begin{lemma}
\Change{Assume $\lambda_3$ has a density with respect to Lebesgue measure, $\frac{d \lambda_3}{d\ell}=\rho_3^\Dl$}. Then, $\nu_3$ has a density with respect to Lebesgue, given by $\frac{d \nu_3}{d\ell} = \rho^X$, where 
 \[
 \rho^X(y)= \tfrac12 \lt(\rho^D * \rho_3^\Dl \lt(\frac{\cdot}{2} \rt) \rt)(3y) +
 (\rho^D * \rho_3^\Dl ) (3y) +
 \lt( (\rho^D * \rho_3^\Dl) (2 \cdot) \rt)\lt (\frac{3y}{2}\rt),
 \]
 and  $\rho^D$ is the pdf of the asymptotic distribution of $D_{12}^n=\bar X_1^{n} - \bar X_2^{n}$, found in Theorem~\ref{lem:diffs}.
\end{lemma}
\begin{proof} 
Notice that the random variables $D_{23}^{n}=\bar X_2^{n} - \bar X_3^{n}$ and $\frac{\Delta_{n+1}}{\sqrt{3}}$ are independent, the pdf of $2\frac{\Delta_{n+1}}{\sqrt{3}}$ is given by  $\tfrac12 \rho_3^\Dl (\frac{\cdot}{2})$ and\Change{, through \eqref{eq:nuDiffsFormula},}
Theorem~\ref{lem:diffs} ensures the asymptotic distribution of $D_{23}^{n}$ has a pdf, $\rho^D$. Thus,
the asymptotic distribution of $\frac{1}{3} \lt(\bar X_2^{n} - \bar X_3^{n} + 2\frac{\Delta_{n+1}}{\sqrt{3}} \rt)$
has pdf given by $h_1(y):=3  \lt(\rho^D  * \tfrac12 \rho_3^\Dl (\frac{\cdot}{2}) \rt) (3y)$.
By symmetry, $h_1$ is also the pdf of $\frac{1}{3} \lt(\bar X_3^{n} - \bar X_2^{n} + 2\frac{\Delta_{n+1}}{\sqrt{3}} \rt)$.

Similarly, and recalling that $\Dl$ is symmetric, the pdf of the asymptotic distribution of 
$ \frac{1}{3} \lt(\bar X_1^{n} - \bar X_3^{n} - \frac{\Delta_{n+1}}{\sqrt{3}} \rt)$ and
$\frac{1}{3} \lt(\bar X_1^{n} - \bar X_2^{n} - \frac{\Delta_{n+1}}{\sqrt{3}} \rt) $ is $h_2(y):= 3 (\rho^D * \rho_3^\Dl) (3y)$, and  that of $\frac{1}{3} \lt(2(\bar X_1^{n} - \bar X_3^{n}) - \frac{\Delta_{n+1}}{\sqrt{3}} \rt)$ and
$\frac{1}{3} \lt(2(\bar X_1^{n} - \bar X_2^{n}) - \frac{\Delta_{n+1}}{\sqrt{3}} \rt)$ is
$h_3(y):=3 (\frac{1}{2} \rho^D(\frac{\cdot}{2}) \ast \rho_3^\Dl) (3y)=3\lt( \rho^D * \rho_3^\Dl (2 \cdot) \rt)\lt (\frac{3y}{2}\rt)$.

Since the indices $(i_{n+1}, j_{n+1})$ are uniformly distributed, we conclude that $ \rho^X=\frac{1}{3}(h_1+h_2+h_3)$, as claimed.
\end{proof}

In fact, in this case we are also able to understand the asymptotic distribution of the entire process $(\bar \X_3^n)_{n\in \N}$, namely by characterising the asymptotic distribution of $(D_{12}, D_{13}):=\lim_{n\to \infty} (D^n_{12}, D^n_{13})$ on $\R^2$. Note that $D^n_{23}=D^n_{13}-D^n_{12}$, so\Change{, arguing as in section~\ref{S:2part},} $(\bar \X_3^n)_{n\in \N}$ is characterised by $(D^n_{12}, D^n_{13})$
through \eqref{equ_Xreconstruct}.

\begin{lemma}\label{lm:rho3dl2}
The asymptotic distribution of $(D^n_{12}, D^n_{13})$, as $n \to \infty$, is characterised by the pdf
\begin{align*}
	\rho^{D,2}(y_1,y_2) = &\tfrac16 \Big( 
	\rho^D(y_1-y_2)\rho_3^\Dl(y_1)+\rho^D(y_1-y_2)\rho_3^\Dl(y_2)
	+\rho^D(y_2)\rho_3^\Dl(y_1)  \\
	&+\rho^D(y_2) \rho_3^\Dl(y_1-y_2) +\rho^D(y_1)\rho_3^\Dl(y_2) +\rho^D(y_1)\rho_3^\Dl(y_1-y_2) \Big), \quad y_1, y_2\in \mathbb R.
\end{align*}
\end{lemma}

\begin{proof}
Specialising \eqref{eq:diffs} we obtain a recurrence scheme for the law of  $(D^n_{12}, D^n_{13})$,
\begin{equation}\label{eq:3partDiffs}
	(D_{12}^{n+1}, D_{13}^{n+1})=\begin{cases}
(0, D^n_{13}-D^n_{12}) +  \frac{\Delta_{n+1}}{\sqrt{3}} (1,1) & i_{n+1} = 1, \  j_{n+1}=2  \\
(D^n_{12}-D^n_{13}, 0) +  \frac{\Delta_{n+1}}{\sqrt{3}} (1,1)   & i_{n+1} = 1, \  j_{n+1}=3  \\
(0, D^n_{13}) +  \frac{\Delta_{n+1}}{\sqrt{3}} (-1,0)  & i_{n+1} = 2, \  j_{n+1}=1  \\
(D^n_{13}, D^n_{13}) +  \frac{\Delta_{n+1}}{\sqrt{3}} (-1,0)   & i_{n+1} = 2, \  j_{n+1}=3  \\
(D^n_{12}, 0) +  \frac{\Delta_{n+1}}{\sqrt{3}} (0,-1) & i_{n+1} = 3, \  j_{n+1}=1  \\
(D^n_{12}, D^n_{12}) +  \frac{\Delta_{n+1}}{\sqrt{3}} (0,-1)  & i_{n+1} = 3, \  j_{n+1}=2  \\
	\end{cases}.
\end{equation}
Notice that $D^n_{13}-D^n_{12}=D^n_{23}$, so we can identify the asymptotic distribution of each of the two  terms in all  the cases above, as $n\to \infty$. For instance, the asymptotic distribution of
$(0, D^n_{13}-D^n_{12})$ as a measure on $\R^2=\{(y_1,y_2): y_1,y_2\in \R \}$
is supported on the $y_2$ axis and can be characterised by the density $\rho^D(y_2) \delta (y_1)$ where $\delta$ denotes the Dirac $\delta$-distribution.
Also, the distribution of  $\frac{\Delta_{n+1}}{\sqrt{3}} (1,1)$ is supported on the diagonal $\{(y,y): y\in\R\}\subset \R^2$ and can be adequately described by the density $\rho_3^\Dl(y_1) \delta (y_1-y_2)$.
Since the terms are independent, the distribution of $(0, D^n_{13}-D^n_{12}) +  \frac{\Delta_{n+1}}{\sqrt{3}} (1,1)$ is the convolution of the two measures. A simple computation shows that this is an absolutely continuous measure with density given by $\rho^D(y_1-y_2)\rho_3^\Dl(y_1)$.
The remaining cases are computed similarly.
\end{proof}

To investigate the case of more particles, it is convenient to consider characteristic functions.
We now present the corresponding analysis for $N=3$.
\begin{cor}\label{cor:phi3}
The characteristic function of $(D_{12}, D_{13})$ is 
\begin{equation}\label{eq:phi3dl2}
\begin{split}
\psi_3^{2}&(s_1,s_2) = \frac16 \Big[
\psi_3(s_1)\lt( \charD \lt(\tfrac{s_2}{\sqrt{3}} \rt) +\charD \lt(\tfrac{s_1+s_2}{\sqrt{3}}\rt) \rt) \\
  + &\psi_3(s_2) \lt( \charD \lt(\tfrac{s_1}{\sqrt{3}} \rt) +\charD \lt(\tfrac{s_1+s_2}{\sqrt{3}}\rt) \rt)
+\psi_3(s_1+s_2)
\lt( \charD \lt(\tfrac{s_1}{\sqrt{3}} \rt)+ \charD \lt(\tfrac{s_2}{\sqrt{3}} \rt) \rt)
\Big],
\end{split}
\end{equation}
where \Change{$s_1, s_2\in \mathbb R$,} $\psi_3(s)=\frac{ \charD(s/\sqrt{3})}{2 - \charD(s/\sqrt{3})}$ is as in \eqref{eq:nuDiffsFormula},  
and the characteristic function of $\bar X_1$ is 
\begin{equation}\label{eq:phi3}
\begin{split}
\phi_3(s)= \frac13 \lt[
\psi_3(s/3) \lt( \charD \lt(\tfrac{s}{3\sqrt{3}} \rt) +\charD \lt(\tfrac{2s}{3\sqrt{3}}\rt) \rt) 
+ \psi_3(2s/3)
\charD \lt(\tfrac{s}{3\sqrt{3}} \rt)
\rt].
\end{split}
\end{equation}

\end{cor}
\begin{proof}
Let $\psi_{3,n}^{2}$ be the characteristic function of  $(D^n_{12}, D^n_{13})$.
As in the proof of Lemma~\ref{lm:rho3dl2},
 we consider the six cases in  \eqref{eq:3partDiffs}. 
For instance, using independence, the characteristic function of
$(0, D^n_{13}-D^n_{12}) +  \frac{\Delta_{n+1}}{\sqrt{3}} (1,1)=(0, D^n_{23}) +  \frac{\Delta_{n+1}}{\sqrt{3}} (1,1)$ is
 $$
 \E \lt(e^{i (s_1,s_2) \cdot  (0, D^n_{23})} \rt)
  \E\lt(e^{i (s_1,s_2) \cdot  \tfrac{\Delta_{n+1}}{\sqrt{3}} (1,1) }\rt)
 =\psi_{3,n}(s_2) \charD(\tfrac{s_1+s_2}{\sqrt{3}}),
 $$
where $\psi_{3,n}$ is the characteristic function of  $D^n_{12}$ (which, by symmetry, coincides with that of $D^n_{23}$). 
Arguing similarly for the remaining cases in \eqref{eq:3partDiffs}, it follows that
\begin{align*}
6\psi_{3,n+1}^{2} (s_1,s_2) =
& \psi_{3,n}(s_1) \lt( \charD \lt(\tfrac{s_2}{\sqrt{3}} \rt) +\charD \lt(\tfrac{s_1+s_2}{\sqrt{3}}\rt) \rt)
  + \psi_{3,n}(s_2)  \lt( \charD \lt(\tfrac{s_1}{\sqrt{3}} \rt) +\charD \lt(\tfrac{s_1+s_2}{\sqrt{3}}\rt) \rt)\\
&+ 
\psi_{3,n}(s_1+s_2) 
\lt( \charD \lt(\tfrac{s_1}{\sqrt{N}} \rt)+ \charD \lt(\tfrac{s_2}{\sqrt{N}} \rt) \rt).
\end{align*}
By Theorem~\ref{lem:diffs}, $\lim_{n\to\infty}\psi_{3,n}(s)=\psi_3(s)$ and \eqref{eq:phi3dl2} follows. 
Since $\bar X_1=\tfrac13 (D_{12}+ D_{13})$, we have that $\phi_3(s) = \psi_3^{2}(s/3,s/3)$ and  \eqref{eq:phi3} follows.
\end{proof}
\section{The general case of $N$ particles}
\label{sec_N}

In this section, we investigate the asymptotic distribution of the process $(\bar \X_N^n)_{n\in \N}$ for arbitrary $N\geq 2$. 
More precisely, we recursively (in $k$) characterise the marginal distribution of $k$ inter-particle distances, for $1\leq k < N$, in the limit as $n\to \infty$.
Recall that $\charD$ denotes the characteristic function of $\Dl$.
We point out that, unlike Section~\ref{sec_simple}, the arguments in this section do not require that $\Dl$ has a pdf.  

\begin{thm}[Marginal distribution of $k$ inter-particle distances]\label{thm:margDiffDist}
Let $N\geq 2$. Let $\bar \X_N^n=(\bar X_1^n, \dots, \bar X_N^n)$ be as in \eqref{eq:process} and  $D_{1j}^{n} = \bar X_1^{n}-\bar X_j^{n}$, for $j=2, \dots, N$. For each $1\leq k<N$, let $\mu_{N,n}^{k}$ be the distribution of $(D_{12}^{n}, \dots, D_{1(k+1)}^{n})$. 
Then, $\mu_N^{k}:=\lim_{n\to\infty} \mu_{N,n}^{k}$ exists as a weak limit of measures on $\R^k$, and its
 characteristic function is recursively defined by $\psi_N^{0}\equiv 1$, and
\begin{equation}\label{eq:phiNdlk}
\psi_N^{k}(s_1, \dots, s_k)=\frac{\chi_N^{k}(s_1, \dots, s_k)}
{(k+1)(N-1) + (k+1-N) \lt( \charD \lt(\tfrac{\sum_{j=1}^k s_j}{\sqrt{N}}\rt)+ \sum_{j=1}^k\charD \lt(\tfrac{\Change{s}_j}{\sqrt{N}}\rt) \rt)},\end{equation}
where \Change{$(s_1, \dots, s_k)\in \mathbb R^k$,}
\begin{equation}\label{eq:chiNdlk}
\begin{split}
\chi_N^{k}(s_1, \dots, s_k)=&
\sum_{i=1}^k\psi_N^{k-1}(s_1, \dots, \hat{s_i},\dots, s_k) \lt( \charD \lt(\tfrac{\Change{s}_i}{\sqrt{N}} \rt) +\charD \lt(\tfrac{\sum_{j=1}^k s_j}{\sqrt{N}}\rt) \rt)\\
&+ 
\sum_{i,j=1,  i\neq j}^k\psi_N^{k-1}(s_1, \dots, \hat{s_i}, s_j+s_i, \dots, s_k) \charD \lt(\tfrac{\Change{s}_i}{\sqrt{N}} \rt),
\end{split}
\end{equation}
 $(s_1, \dots, \hat{s_i},\dots, s_k)$ denotes the $k-1$ tuple obtained  by removing the $i^{th}$ entry from $(s_1, \dots, s_k)$, and
$(s_1, \dots, \hat{s_i}, s_j+s_i, \dots, s_k)$ is the $k-1$ tuple
 $(s_1, \dots, \hat{s_i},\dots, s_k)$, where the entry corresponding to $s_j$ is replaced by $s_j+s_i$.
\end{thm}
\begin{proof}
The proof proceeds inductively on $k$.
The case $k=1$ has been addressed in Theorem~\ref{lem:diffs}, and one can directly check that the formula for $\psi_N^{1}$ given in \eqref{eq:phiNdlk} 
agrees with \eqref{eq:CFDiffsFormula}.

Now let $1 < k < N$. Assuming the result is known for some $k - 1$, we establish it for $k$.
The evolution of $\cD_{N,n}^k:=(D^n_{12}, D^n_{13}, \dots, D^n_{1(k+1)})$ is described in the following table.
The third column denotes the generic $D^{n+1}_{1\ell}$, for $\ell \in\{2,3, \dots,k+1\}$.
The fourth column shows the characteristic function of $\cD_{N,n}^k$, conditioned on the choice of $\Change{(i,j):=}(i_{n+1},j_{n+1})$,
 which is obtained arguing as in the proof of Corollary~\ref{cor:phi3}, recalling that there is symmetry under index permutations, and that $\Dl_{n+1}$ is independent of $\cD_{N,n}^k$. 
\begin{center}
	\begin{tabular}{|c|c||c|c|c||c|c|c|c|ccc}
		\hline
		$\Change{i}$& $\Change{j}$ &$D^{n+1}_{1\ell} \, , \; 2 \leq \ell \leq k+1$
		& $\psi_{N,n}^{ k}(s_1, \dots, s_k)$   \\ 
		\hline
		1& $2 \ldots k+1$ & $\begin{cases} D^{n}_{j\ell} + \Dl_{n+1} \, , & \ell \neq j  \\ \Dl_{n+1} & \ell = j  \end{cases}$
		 &$\psi_{N,n}^{ k-1}(s_1, \dots, \hat{s}_{j-1},\dots, s_k) \charD(\tfrac{1}{\sqrt{N}}\sum_{\ell=1}^k s_\ell)$\\
		$2 \ldots k+1$& 1 &$\begin{cases} D^{n}_{1\ell}  \, , & \ell \neq i  \\ -\Dl_{n+1} & \ell= i  \end{cases}$
		&$\psi_{N,n}^{ k-1}(s_1, \dots, \hat{s}_{i-1},\dots, s_k) \charD(\tfrac{\Change{s}_{i-1}}{\sqrt{N}})$ \\
		\multicolumn{2}{|c||} {$2 \ldots k+1\, , \; i \neq j$} &$\begin{cases} D^{n}_{1\ell} \, , & \ell \neq i  \\ D^{n}_{1j}-\Dl_{n+1} & \ell = i  \end{cases}$
		 &$\psi_{N,n}^{ k-1}(s_1, \dots, \hat{s}_{i-1}, s_{j-1}+s_{i-1}, \dots, s_k) \charD(\tfrac{\Change{s}_{i-1}}{\sqrt{N}})$   \\
		\hline
		1& $k+2 \dots N$ & $D^{n}_{j\ell} + \Dl_{n+1}$ 
		& $\psi_{N,n}^{ k}(s_1, \dots, s_k) \charD(\tfrac{1}{\sqrt{N}}\sum_{\ell=1}^k s_\ell)$  \\
		$2 \ldots k+1$& $k+2 \dots N$ &$\begin{cases} D^{n}_{1\ell}  \, , & \ell \neq i \\ D^{n}_{1j}-\Dl_{n+1} & \ell= i  \end{cases}$
		& $\psi_{N,n}^{ k}(s_1, \dots, s_k) \charD(\tfrac{\Change{s}_{i-1}}{\sqrt{N}})$  \\
		$k+2\dots N$& $ 1 \ldots N\,, \; j \neq i$ &$D^{n}_{1\ell}$ 
		  & $\psi_{N,n}^{ k}(s_1, \dots, s_k)$ \\
		\hline
	\end{tabular}
\end{center}

Thus, considering all the cases for $(i_{n+1},j_{n+1})$, one gets an expression for $\psi_{N,n+1}^{ k}$
 in terms of $\psi_{N,n}^{ k-1}$, $\psi_{N,n}^{ k}$ and $\charD$, as follows,
\begin{multline}
	\label{eq:phiNndlk}
	\psi_{N,n+1}^{ k} (s_1, \dots, s_k)=\frac{1}{N(N-1)} \Biggl[\sum_{i=1}^k\psi_{N,n}^{k-1}(s_1, \dots, \hat{s_i},\dots, s_k) \lt( \charD \lt(\tfrac{\Change{s}_i}{\sqrt{N}} \rt) +\charD \lt(\tfrac{\sum_{j=1}^k s_j}{\sqrt{N}}\rt) \rt)\\
	+ 
	\sum_{i,j=1,  i\neq j}^k\psi_{N,n}^{k-1}(s_1, \dots, \hat{s_i}, s_j+s_i, \dots, s_k) \charD \lt(\tfrac{\Change{s}_i}{\sqrt{N}} \rt) +\\
	+ \psi_{N,n}^{ k} (s_1, \dots, s_k)\times
 (N-k-1)\left((N-1) 
	+ \lt( \charD \lt(\tfrac{\sum_{i=1}^k s_i}{\sqrt{N}}\rt)+ \sum_{\ell=1}^k\charD \lt(\tfrac{\Change{s}_\ell}{\sqrt{N}}\rt) \rt)  \right) \Biggr] \; .
\end{multline}
 
To obtain the convergence result as $n \rightarrow \infty$, note that for each fixed $(s_1,\dots,s_k)$, \eqref{eq:phiNndlk} has the structure of a linear recurrence relation, $\psi_{N,n+1}^{ k}=A \psi_{N,n}^{ k} +B_n$ where $A<1$ and by our induction assumption $B_n$ converges to a limit which we will denote by $B_\infty$. Let $\psi_{N}^{ k}$ be the solution of $\psi=A \psi+B_\infty$. Then, by linearity, $\Delta \psi_{N,n}^k := \psi_{N,n}^{ k}-\psi_{N}^{ k}$ satisfies the  relation $\Delta \psi_{N,n+1}^{ k}=A \Delta \psi_{N,n}^{ k} +\Delta B_n$ where $\Delta B_n=B_n-B_\infty \rightarrow 0$ as $n \rightarrow \infty$. Its solution can be written as $\Delta \psi_{N,n+1}^{ k}=A^{n+1} \Delta \psi_{N,0}^{k} +\sum _{m=0}^n \Delta B_n A^{n-m}$ which converges to 0 as $n \rightarrow \infty$.
As a consequence, the limit as $n\to \infty$ of $\psi_{N,n}^{ k}$ 
satisfies \eqref{eq:phiNdlk}. 
\Change{Weak convergence of $\mu_{N,n}^k$ to $\mu_{N}^k$ follows from L\'evy's continuity theorem.}
\end{proof}

\begin{cor}[Single particle distribution]
\label{cor_phiN}
\Change{Under the hypothesis of Theorem~\ref{thm:margDiffDist}},
the characteristic function of the limit distribution of a single particle, e.g. $\bar X_1^n$, as $n\to\infty$,  is given by
\Change{$\phi_{N}:\mathbb R \to \mathbb C$},
\[
\phi_{N}(s)  = \psi_{N}^{N-1}(s/N,\dots,s/N) \; .
\]
\end{cor}
\begin{proof}
	Notice that, since $\sum_{j=1}^N \bar X_{j}^n=0$, it holds that
	$\bar X_1^n = \frac1N\sum_{j=2}^N D_{1j}^n$.
	Recalling that $\phi_{N,n}(s)$ is the characteristic function of $\bar X_1^n$ and  $\psi_{N,n}^{N-1}(s_1,\dots,s_{N-1})$
	is the characteristic function of $(D_{12}^n, \dots, D_{1N}^n)$, we obtain 
	\[
	\phi_{N,n}(s) = \E( e^{i s \bar X_1^n}) =  \E( e^{i  \frac sN\sum_{j=2}^N D_{1j}^n }) = \psi^{N-1}_{N,n}(s/N,\dots,s/N) \; .
	\]
Using Theorem~\ref{thm:margDiffDist} to take the limit as $n \rightarrow \infty$ implies the result.
\end{proof}
\section{Large population limit ($N \rightarrow \infty$)}
\label{equ_Nlimit}
In this section, we investigate the limit as the population size $N$ grows to infinity. 
First, we characterise the limit of marginal distributions of $k$ inter-particle distances in Theorem~\ref{thm:limitDiffDist}, for $k>1$. Then, we investigate the limit of the single particle distributions. In Theorem~\ref{thm:limPartDist}, we identify the corresponding characteristic function. In Section~\ref{sec:moments}, we provide an algorithm to compute its associated moments.

\subsection{Limit marginal of
joint inter-particle distance distributions}
\begin{thm}[Limit marginal distribution of $k$ inter-particle distances]
\label{thm:limitDiffDist}
Let\footnote{The case $k=1$ is the content of \S\ref{S:1diff}, yielding $\psi_\infty^{1}(s)=\frac{2}{2+\sg^2s^2}$.} $k>1$.
Then, the  limit $\mu_\infty^{k}:=\lim_{N\to \infty} \mu_N^{k}$ exists as a weak limit of measures on $\R^k$, and its
 characteristic function is recursively defined by
\begin{multline}\label{eq:phiInfdlk}
\psi_\infty^{k}(s_1, \dots, s_k)=\\=\frac{2\sum_{i=1}^k\psi_\infty^{k-1}(s_1, \dots, \hat{s_i},\dots, s_k) +\sum_{i,j=1,  i\neq j}^k\psi_\infty^{k-1}(s_1, \dots, \hat{s_i}, s_j+s_i, \dots, s_k) }
{k(k+1)+\frac{\sg^2}{2}(s_1^2+\dots+s_k^2 + (s_1+\dots + s_k)^2) },
\end{multline}
where \Change{$(s_1, \dots, s_k)\in \mathbb R^k$,}
$(s_1, \dots, \hat{s_i},\dots, s_k)$ denotes the $k-1$ tuple obtained  by removing the $i^{th}$ entry from $(s_1, \dots, s_k)$, and
$(s_1, \dots, \hat{s_i}, s_j+s_i, \dots, s_k)$ is the $k-1$ tuple
 $(s_1, \dots, \hat{s_i},\dots, s_k)$, with the entry corresponding to $s_j$ replaced by $s_j+s_i$.
In particular, $\psi_\infty^{k}$ is $C^\infty$.
\end{thm}
\begin{proof}
First, notice that  for every $s\in\R$, $\charD (s/\sqrt{N})=1-\tfrac{\sg^2s^2}{2N}+o(\tfrac{s^2}{N})$, as $N\to\infty$.
We will show that the numerator and denominator in  \eqref{eq:phiNdlk} converge to numerator and denominator in  \eqref{eq:phiInfdlk}, as $N\to\infty$. 

Assuming by induction the existence of $\psi_\infty^{k-1}$,
the numerator in  \eqref{eq:phiInfdlk} coincides with $\lim_{N\to\infty} \chi_N^{k}(s_1, \dots, s_k)$, where $\chi_N^{k}$ is defined in Theorem~\ref{thm:margDiffDist}. 
The denominator in \eqref{eq:phiNdlk} is given by
\begin{align*}
&(k+1)(N-1) + (k+1-N) \lt( k+1-  \frac{\sg^2}{2N} \lt( \lt(\sum_{j=1}^k s_j\rt)^2 +\sum_{j=1}^k s_j^2\rt)+o\lt(\frac{\sum_{j=1}^k s_j^2}{N}\rt)\rt)\\
&=k(k+1)-\frac{(k+1-N)\sg^2}{2N}\lt( \lt(\sum_{j=1}^k s_j\rt)^2 + \sum_{j=1}^k s_j^2\rt) + o \lt(\sum_{j=1}^k s_j^2 \rt),
\end{align*}
as $N\to\infty$. Taking the limit as $N\to\infty$ yields the denominator in \eqref{eq:phiInfdlk}, and completes the proof of \eqref{eq:phiInfdlk}. 

Note that for $k <N$, the computations of $\psi^k_\infty(s_1, \dots, s_k)$ through \eqref{eq:phiInfdlk} and $\psi^k_N(s_1, \dots, s_k)$  through \eqref{eq:phiNdlk} involve $k$ recursion steps which on each level branch out $k(k+1)$ times --- thus through an algorithm which involves summation and composition of a finite number of mappings, all of which converge as $\change{N} \rightarrow \infty$. This implies the pointwise convergence $\psi^k_N \rightarrow \psi^k_\infty$ and, \Change{by L\'evy's continuity theorem}, the weak convergence $\mu_N^{k}\rightarrow \mu_\infty^{k}$.

The final claim follows by induction on $k$, starting from the fact that $\psi_\infty^{1}(s)=\frac{2}{2+\sg^2s^2}$ is $C^\infty$.
\end{proof}

\subsection{Limit particle distribution}
Now, our goal is to pass to the limit $N \rightarrow \infty$ in the reconstruction of the single particle distribution from the inter-particle distances in Corollary~\ref{cor_phiN}. 
 In Theorem~\ref{thm:limPartDist}, we show this limit exists. In Algorithm~\ref{alg:cap}, we show a way to compute the moments of this distribution. We illustrate the agreement with numerical simulations in Figure~\ref{fig:empirical}.

\subsubsection{Existence and characterisation of the limit}
The main result of this section is the following.
\begin{thm}[Limit particle distribution]\label{thm:limPartDist}
	The  limit $\nu_\infty:=\lim_{N\to \infty} \nu_N$ exists as a weak limit of measures on $\R$, and its
	characteristic function is given by 
	\begin{equation}
		\label{equ_phiinf}
		\phi_\infty(s)=\lim_{N\to\infty}
		\psi_N^{N-1}(s/N, \dots, s/N), \Change{\quad s\in \mathbb R,}
	\end{equation}
	where $\psi_N^{N-1}$ is as in Theorem~\ref{thm:margDiffDist}.
\end{thm}
This theorem follows directly from Theorem~\ref{it:p5}
and Lemma~\ref{lem:equivConvergence}. 
The scheme of the proof is as follows.
To investigate the limit distribution of particles \eqref{equ_phiinf}, we must
consider the limit as $N\to\infty$ of
\begin{equation*}
\tilde\gm_N(s):=\psi_N^{N-1}(s/N, \dots, s/N)\Change{, \quad s\in \mathbb R} \; .
\end{equation*} 
To do this,  we will, in Theorem~\ref{it:p5}, prove the pointwise convergence of the family of auxiliary functions
\begin{equation}
\label{def_gammaN}
\gm_N(s):=\psi_\infty^{N-1}(s/N, \dots, s/N) \; \Change{, \quad s\in \mathbb R},
\end{equation}
as $N \rightarrow \infty$, where $\psi_\infty^{N-1}$ is as in Theorem~\ref{thm:limitDiffDist}. 
This is done using the method of moments, e.g. \cite[Section 30]{Billingsley}. 
To find moments of the distributions involved, we recall that the $j^{th}$ moment of a random variable with characteristic function $\gm_N(s)$ is given by $i^{-j}  \frac{d^j \gm_N(s) }{ds^j}|_{s=0}$. 
The limits of these moments as $N\to\infty$ are computed in Lemma~\ref{it:p1}, relying on three auxiliary lemmas: Lemma~\ref{it:p2}, providing a simple relation among various partial derivatives of $\phi_\infty^k$, for different $k$;
Lemma~\ref{it:p4}, providing a recursive scheme
to evaluate partial derivatives; and
Lemma~\ref{lem_Phibound}, establishing the necessary growth bound for the method of moments to apply.

To conclude, in Lemma~\ref{lem:equivConvergence}, we will prove that  $\{\tilde\gm_N(s)\}_{N\in\N}$ converges pointwise if and only if $\{\gm_N(s)\}_{N\in\N}$ does, and in this case their limit coincide. 

We proceed by introducing some notation.
For every multi-index $\al=(\al_1, \dots, \al_j) \in \mathbb{N}_0^j$,
let $|\al |=\sum_{i=1}^j \al_i$ and $\#\al$ the number of non-zero elements of $\al$.
Also, let $\al!=\al_1!\dots\al_j!$ and $\partial_\al=\partial_{s_1}^{\al_1}\dots \partial_{s_j}^{\al_j}$, using the convention that $\partial^0_{s_i}$ is the identity for every $1\leq i <N$.  
For multi-indices $\al,\bt \in \mathbb{N}_0^j$, we write $\bt\leq \al$ if $\bt_i\leq \al_i$ for every $1\leq i\leq j$,
and in this case we define $\binom{\al}{\bt}=\frac{\al!}{\bt! (\al-\bt)!}$.

For a multi-index $\alpha \in \mathbb{N}_0^k$ we define the pruning operator $\alpha \mapsto \langle \al \rangle \in \mathbb{N}^{\# \alpha}$ which omits all zero indices from $\alpha$ and returns an ordered version of it, 
\begin{equation}
	\label{equ_pruning}
 \Change{\langle (\alpha_1, \alpha_2, \dots,  \alpha_{k}) \rangle =
	(\al_{i_1}, \al_{i_2}, \dots, \al_{i_{\# \al}})}
\end{equation}
where $(i_j)_{1 \leq j \leq k}$ is a permutation of $1,...,k$, $\al_{i_1} \geq  \al_{i_{\Change{2}}} \geq \dots \geq \al_{i_{\# \al}}>0$ and $\alpha_{i_j}=0$, for $\#\alpha < j \leq k$.
\Change{E.g., $\langle (1,0,3,1) \rangle =
	(3,1,1)$.}

In what follows, for $k \in \mathbb{N}$ and $\alpha \in \mathbb{N}^k$ let 
\begin{equation}
\label{def_Phi}
\Phi_\al:= \partial_{\al} \psi_\infty^k(0, \dots, 0) \; .    
\end{equation}

\begin{lemma} 
	\label{it:p2}
Recall from Theorem~\ref{thm:limitDiffDist} that $\psi_\infty^{k}$ is $C^\infty$ for every $k$.
We claim that for every multi-index $\al \in \mathbb{N}_0^k$ 
we have that 
$$\partial_{\al} \psi_{\infty}^{k}(0, \dots, 0) = \partial_{\langle \al \rangle} \psi_{\infty}^{\#\al}(0, \dots, 0) = \Phi_{\langle \al \rangle}\;,$$
where $\alpha \mapsto \langle \alpha \rangle$ is the pruning operator \eqref{equ_pruning}. 
\end{lemma}

\begin{proof}
First, note that for $1\leq j <k$  \Change{and $(s_1, \dots, s_{j})\in \mathbb R^j$} it holds that
\begin{equation}\label{eq:pruneDer}
\psi_\infty^{k}(s_1, \dots, s_j, 0, \dots, 0)=\psi_\infty^{j}(s_1, \dots, s_{j})\; .
\end{equation}
This claim can be verified directly from the definition of characteristic functions $\psi_\infty^k$ and $ \psi_\infty^j$.
Then, the lemma follows from
the fact that, by the symmetries of the process, the quantity on the left-hand side \Change{of} \eqref{eq:pruneDer} is invariant under permutation of the arguments. 
\end{proof}

\begin{lemma}
	\label{it:p4}
	For $k \in \mathbb{N}$ and for every strictly positive multi-index $\al \in \mathbb{N}^k$ it holds that
	\[
	k(k+1)  \Phi_\al= 
	\sum_{i \neq j} \Phi_{\al_{i \rightarrow j}} -  \sum_{\substack{\bt\leq \al\\\ |\bt|=2, \#\bt=2}} \binom{\al}{\bt} \sg^2 \Phi_{\langle{\al-\bt} \rangle} -\sum_{\substack{\bt\leq \al\\\ |\bt|=2, \#\bt=1}} \binom{\al}{\bt} 2\sg^2 \Phi_{\langle{\al-\bt} \rangle} \;,
	\]
	where $\al_{i \rightarrow j} \in \mathbb{N}^{k-1}$ is the multi-index $(\al_1, ..., \hat \al_i, ..., \al_j+\al_i,..., \al_k)$ obtained from $\alpha \in \mathbb{N}^k$  by removing index $i$ and adding it to index $j$. 
\end{lemma}
\begin{proof}
	Let $\sN_\infty(s_1, \dots, s_k)$ and $\sD_\infty(s_1, \dots, s_k)$ be the numerator and denominator of \eqref{eq:phiInfdlk}, respectively. 	Theorem \ref{thm:limitDiffDist} shows that 
	\begin{equation}
		\label{equ_limitrecursion}
		\psi_\infty^{k}(s_1, \dots, s_k) \sD_\infty(s_1, \dots, s_k)=\sN_\infty(s_1, \dots, s_k) \; . 
\end{equation}		
	We note that for every multi-index $\bt \in \mathbb{N}_0^k$,  due to the specific structure of the denominator $\mathcal{D}_\infty$, it holds that
	\begin{equation*}
		\partial_\bt\sD_\infty(0,\dots,0)=\begin{cases}
			k (k+1) & |\bt|=0 \; ,\\
			2 \sigma^2 & |\bt|=2\, , \; \#\bt=1\; ,\\
			 \sigma^2 & |\bt|=2\, , \; \#\bt=2 \; , \\
			0 & \text{otherwise} \; .
		\end{cases}
	\end{equation*}
	On the other hand, for the numerator $\mathcal{N}_\infty$ we find that  $\partial_\al \sN_\infty(0,\dots,0) = \sum_{i \neq j} \Phi_{\al_{i \rightarrow j}}$ since the order of all $k$ partial derivatives in $\partial_\alpha \mathcal{N}_\infty$ is non-zero, which implies that
	$$\partial_\al \lt(\sum_{i=1}^{k}\psi_\infty^{k-1}(s_1, \dots, \hat{s_i},\dots, s_{k})\rt)  \Biggr|_{\vec{s}=0}=0 \; .$$
	Using these identities when taking the derivative $\partial_\alpha$ of \eqref{equ_limitrecursion} we obtain the result by the Leibniz rule and Lemma~\ref{it:p2}.
\end{proof}

\begin{lemma}
\label{lem_Phibound}
Let $\al \in \mathbb{N}^k$ and $j=|\al|$. Then, $| \Phi_\al| \leq (\sg^2/2)^{j/2} j!=:M_j$~. In particular it holds that  $|\Phi_{\one_j}| \leq  (\sg^2/2)^{j/2} j!$, \Change{where $\one_j=(1,\dots,1) \in \mathbb{N}^j$. }.   
\end{lemma}
\begin{proof}
We proceed by induction on $j$. The result holds true for $M_0=1$. Consider an integer $j\geq 1$  and let ${\bt_j}=(j)$.
Then, $|\Phi_{\bt_j}|= |\partial_{s}^j \psi_\infty^{1}(s)|_{s=0}$, by  ~\ref{it:p2}.
By Corollary~\ref{cor:Laplace} and properties of characteristic functions, these derivatives are given by the moments of the Laplace distribution with characteristic function $\frac{2}{2+s^2\sigma^2}$. Thus, $|\Phi_{\bt_j}|=0$ if $j$ is odd, and $|\Phi_{\bt_j}|=\frac{j!\sg^j}{2^{j/2}}$ if $j$ is even. Particularly it holds that
$|\Phi_{\beta_j}| = j(j-1) \sigma^2/2 \, |\Phi_{\beta_{j-2}}|$, which can also be written as 
$$j(j+1) |\Phi_{\beta_j}| = j(j-1) |\Phi_{\beta_j}|+ j\binom{j}{2} 2\sigma^2 \, |\Phi_{\beta_{j-2}}| \; .$$
Now, for each $j\geq 0$, let $M_j=\max_{|\al|=j} |\Phi_\al|$. Lemma~\ref{it:p4} and the previous paragraph ensure that for each $j\geq 2$ and for any multindex $\alpha_j$ with $|\alpha_j|=j$ and $\#\alpha_j>1$ it holds that
$$j (j+1) |\Phi_{\alpha_j}| \leq j (j-1) M_j + \binom{j}{2}2 \sg^2 M_{j-2} \; , $$
where we have used the multi-binomial theorem to get $\sum_{\substack{\bt\leq \al\\\ |\bt|=2}} \binom{\al}{\bt} = \binom{j}{2}$. For $M_j$ we obtain that 		
$$j (j+1) M_j \leq j (j-1) M_j + j \binom{j}{2}2 \sg^2 M_{j-2} \; . $$		

Thus $ M_j \leq  \frac{\sg^2}{2}j(j-1)  M_{j-2}\leq  \frac{\sg^2}{2}j(j-1) \frac{\sg^2}{2}(j-2)(j-3)  M_{j-4} \leq \dots$. The recursion implies that $M_j=0$ if $j$ is odd and otherwise $M_j \leq (\sg^2/2)^{j/2}  j!$.
\end{proof}

\begin{lemma}
 \label{it:p1}
\mysout{Let $\one_j=(1,\dots,1) \in \mathbb{N}^j$. Then,}
\Change{For every $j\in \mathbb N$, the following holds,}
$$\lim_{N\to\infty}\frac{d^j}{ds^j}\gm_N(0) = \Phi_{\one_{j}} 
\; .$$ 
\end{lemma}

\begin{proof}
The chain rule yields
$$\frac{d^j}{ds^j}\gm_N(s) = N^{-j}  \sum_{(\ell_1, \dots, \ell_j) \in (1,..., N-1)^j} \partial_{s_{\ell_1}}\dots \partial_{s_{\ell_j}} \psi_\infty^{N-1}(s/N, \dots, s/N).$$
The sum above has $A:= (N-1)^j$ terms, and there are $A_1:=(N-1)(N-2)\dots(N-j)$ of them which agree with $\partial_{\one_j} \psi_\infty^{N-1}(s/N,\dots,s/N)$. 
This holds because partial derivatives of $\psi_\infty^{N-1}$ commute as it is build from smooth functions through a recursive scheme according to Theorem~\ref{thm:margDiffDist}.
As a consequence, there are at most $A_2=A-A_1$ terms which do not agree with this term. 
Furthermore, Lemmas~\ref{it:p2} and \ref{lem_Phibound} ensure there is a bound $M_j$ on all derivatives of order $j$,
$\partial_{s_{\ell_1}}\dots \partial_{s_{\ell_j}} \psi_\infty^{N-1}(0, \dots, 0)$,
which is independent of $N$.
Since $A_2$ is a polynomial in $N$ of order $j-1$ it holds that $A_2=o(N^j)$ and these terms do not contribute to $\frac{d^j}{ds^j}\gm_N(s)$ in the limit $N \rightarrow \infty$. 

Thus, using Lemma~\ref{it:p2} we get
\begin{equation}
   \lim_{N \rightarrow \infty} \frac{d^j}{ds^j} \gamma_N(0)=\lim_{N \rightarrow \infty} \partial_{\one_{j}} \psi_\infty^{N-1}(0, \dots, 0)=\partial_{\one_{j}} \psi_\infty^{j}(0, \dots, 0) =  \Phi_{\one_{j}} \; .
\end{equation}
\end{proof}

\begin{thm}
	\label{it:p5}
	The sequence of functions
	$\{\gm_N \}_{N\in\N}$ defined in \eqref{def_gammaN} converges pointwise for every $s \in \mathbb{R}$ as $N \rightarrow \infty$.
\end{thm}
\begin{proof}
We interpret $\gamma_N$ as characteristic function. Its associated distribution has moments given by $i^{-j} \frac{d^j}{ds^j}\gamma_N(0)$. 
Lemma~\ref{it:p1} shows that
$\lim_{N \rightarrow \infty} \frac{d^j}{ds^j} \gamma_N(0) =  \Phi_{\one_{j}}$.
Since the power series $\sum_{j=0}^\infty \Phi_{\one_j} x^j/j!$ has positive radius of convergence, due to the estimate in Lemma~\ref{lem_Phibound}, the distribution with moments $\Phi_{\one_{j}}$ is unique, according to \cite[Theorem 30.1]{Billingsley}.
As a consequence, 
the distributions associated to $\gamma_N$ converge \mysout{in distribution} \Change{weakly} to a limit $\nu$ according to \cite[Theorem 30.2]{Billingsley}. 
L\'evy's continuity theorem then implies that the characteristic function of $\nu$ is the pointwise limit  $\lim_{N\to\infty}\gm_N(s)$.
\end{proof}

To complete the proof of Theorem~\ref{thm:limPartDist}, we show the following.
\begin{lemma}\label{lem:equivConvergence}
Let $s\in \mathbb R$. Then,
	the sequence $\{\tilde\gm_N(s)\}_{N\in\N}$ converges  if and only if $\{\gm_N(s)\}_{N\in\N}$ does, and in this case their limits coincide.
\end{lemma}
\begin{proof}
The lemma will follow once we show that for every $s\in\mathbb R$,
$$\lim_{N\to\infty}  |\tilde\gm_N(s)-\gm_N(s)| =0.$$

We will first show the following:
For $R>0$ let $B^k_R=\{(t_1, .., t_k)\in \mathbb R^k: |t_1|+\dots+|t_k| \leq R\}$\mysout{(the ball of radius $R$ in the $l^1$-norm of $\mathbb{R}^k$)}, it holds that 
\begin{equation}\label{eq:unifConv}
\overline{\lim}_{N\to\infty} \max_{1\leq k < N} \|\psi_N^{k}-\psi_\infty^{k}\|_{\Change{\infty, B^k_R}}\leq \overline{\lim}_{N\to\infty} \|\psi_N^{1}-\psi_\infty^{1}\|_{\Change{\infty, B^1_R},}
\end{equation}
\Change{where $\|\psi\|_{\infty,S}$ denotes the supremum norm of $\psi$ restricted to $S$.}

	Let $(s_1, \dots, s_k) \in B^k_R$ and let  $\sN_\infty(s_1, \dots, s_{k})$ and $\sD_\infty(s_1, \dots, s_{k})$ be the numerator and denominator of \eqref{eq:phiInfdlk}, respectively. Similarly, let  $\sN_N(s_1, \dots, s_{k})$ and $\sD_N(s_1, \dots, s_{k})$ be the numerator and denominator of \eqref{eq:phiNdlk}, respectively. 
	We note that for every $k< N$ we have that
	$\sD_\infty, \sD_N \geq k(k+1)$ and $|\sN_N|, |\sN_\infty| \leq k(k+1)$, where the latter follows from the fact that there are $k(k+1)$ terms in the sum defining $\sN_N(s_1, \dots, s_{k})$, and each of them has absolute value at most 1. Hence,
	\begin{multline}
		\label{eq:DiffPhiNPhiInf}
		|\psi_N^{k}(s_1,\dots, s_k) -\psi_\infty^{k}(s_1,\dots, s_k)| 
		=  \lt| \frac{\sN_N-\sN_\infty}{\sD_N} + \frac{\sN_\infty   }{\sD_N \sD_\infty} (\sD_\infty-\sD_N) \rt|
		\leq \\
		\leq 	\frac{|\sN_N(s_1, \dots, s_{k})-\sN_\infty(s_1, \dots, s_{k})|}{k(k+1)} +
		\frac{\lt| \sD_N(s_1, \dots, s_{k})-\sD_\infty(s_1, \dots, s_{k}) \rt|}{k(k+1)} \; . 
	\end{multline}
	To bound the first term, observe that the corresponding numerator may be split into 
	$k(k+1)$ terms, each of the form $|\psi_N^{k-1}(t_1,\dots,t_{k-1}) -\psi_\infty^{k-1}(t_1,\dots,t_{k-1})|$ with $(t_1,\dots,t_{k-1}) \in B^{k-1}_R$,
	hence it is bounded by $\| \psi_N^{k-1} -\psi_\infty^{k-1} \|_{\Change{\infty, B_R^{k-1}}}$.
	
	To bound the second term, we argue as in the proof of Theorem~\ref{thm:limitDiffDist}. 
Recall that, for each $s\in \mathbb R$,  $\charD (s/\sqrt{N}) = 1-\tfrac{\sg^2s^2}{2N}+ o(\frac{s^2}{N})$ as $N\to\infty$.
	Furthermore, since $(s_1, \dots, s_k)\in B_R^k$, then
 $\lt(\sum_{j=1}^k s_j\rt)^2 + \sum_{j=1}^k s_j^2\leq 2 \|(s_1, \dots, s_{k})\|_1^2 \leq 2 R^2$.
	Thus, for $N>k$, the following bound holds uniformly on $B_R^k$
	\begin{align*}
		|\sD_N(s_1, \dots, s_{k})-\sD_\infty(s_1, \dots, s_{k})|
		&\leq \frac{(k+1)\sg^2}{2N}\lt( \lt(\sum_{j=1}^k s_j\rt)^2 + \sum_{j=1}^k s_j^2\rt)+ o \lt(\sum_{j=1}^k s_j^2\rt)\\
		&\leq \frac{(k+1)\sg^2}{N} R^2 + o\lt( R^2\rt) \; .
			\end{align*}
	Incorporating this into \eqref{eq:DiffPhiNPhiInf}  and taking the supremum over $B^k_R$ yields 
	\begin{align*}
		\|\psi_N^{k} -\psi_\infty^{k}\|_{\Change{\infty, B^k_R}}
		&\leq \| \psi_N^{k-1} -\psi_\infty^{k-1} \|_{\Change{\infty, B_R^{k-1}}}+\frac{1}{k} \frac{\sg^2}{N} R^2 
  +  \frac{1}{k^2} o\lt(R^2\rt)\; .	
	\end{align*}
	Iterating this argument, and recalling that $\sum_{j=2}^N \frac{1}{j}\leq \log N$ and $\sum_{j=1}^\infty \frac{1}{j^2}<\infty$, for every $k<N$ we obtain
	\begin{align*}
		\|\psi_N^{k} -\psi_\infty^{k}\|_{\Change{\infty, B^k_R}}
		&\leq \| \psi_N^{1} -\psi_\infty^{1} \|_{\Change{\infty, B_R^{1}}}+\frac{\log N }{N} \sg^2 R^2 +  o\lt( R^2\rt),
	\end{align*}
	as $N\to\infty$. Letting $N\to\infty$ yields \eqref{eq:unifConv}.

The proof of Corollory~\ref{cor:Laplace} shows that, as $N\to\infty$, we have that 
$\|\psi_N^{1}-\psi_\infty^{1}\|_{\Change{\infty, B^1_R}}=\mathcal O(R^2/N)$, so the previous paragraph implies that
$\lim_{N\to\infty}  \|\psi_N^{N-1}-\psi_\infty^{N-1}\|_{\Change{\infty, B^{N-1}_R}}=0$.
Thus, for every $s\in \mathbb R$,
$$\overline{\lim}_{N\to\infty}  |\tilde\gm_N(s)-\gm_N(s)| \leq
\lim_{N\to\infty}  \|\psi_N^{N-1}-\psi_\infty^{N-1}\|_{\Change{\infty, B^{N-1}_{|s|}}}=0,$$
and the proof is concluded.
\end{proof}

\subsubsection{Computation of moments of $\nu_\infty^{(1)}$}\label{sec:moments}
To compute the moments of $\nu_\infty^{(1)}$ we use that they are given by $i^{-j} \Phi_{\one_j}$ according to Theorem~\ref{it:p5}. The partial derivatives $\Phi_{\one_j}$ can be computed through the recursive scheme provided in Lemma~\ref{it:p4}, and the derivatives of $\psi^1$ can be directly computed using Corollary~\ref{cor:Laplace}. 

For the upcoming discussion, we recall the concept of partitions: The multi-index $\ell$ is a partition  of the natural number $n$ if it is an ordered tuple $(\ell_1, \ell_2, ..., \ell_m)$ such that $\ell_i \in \mathbb{N}$, $\ell_{i}\geq\ell_{i+1}$ for $1 < i < m$ and $\ell_1+...+\ell_m=n$. Write $|\ell|=n$ and $\#\ell=m$.

Note that the recursion Lemma~\ref{it:p4} relates $\Phi_\alpha$ to other partial derivatives $\Phi_\beta$ of higher order in the canonical ordering of partitions\cite{AbramowitzMilton1972}\footnote{\href{http://oeis.org/wiki/Orderings_of_partitions}{http://oeis.org/wiki/Orderings\_of\_partitions}}
\begin{equation}
\label{equ_ordering}
    \begin{aligned}
&\{()\}\\
&\{(1)\}\\
&\{(2), (1, 1)\}\\
&\{(3), (2, 1), (1, 1, 1)\} \\
&\{(4), (3, 1), (2, 2), (2, 1, 1), (1, 1, 1, 1)\}\\
&\vdots
 \end{aligned}
\end{equation}
It holds that either $|\beta|=|\alpha|-2$ or that $\beta$ is larger than $\alpha$ in the canonical order of partitions of the integer $|\beta|=|\alpha|$, i.e. $\beta$ is left of $\alpha$ in \eqref{equ_ordering}.

For odd-order derivatives of $\Phi_\alpha$ this means that only odd-order derivatives are used to compute them, starting with odd derivatives of $\psi^1$ at $s=0$ which all vanish. Therefore all odd-order derivatives $\Psi_\alpha$ are zero.

As for even-order derivatives we start the recursion with $\Phi_{()}:=\psi^1(0)=1$ and obtain for the partial derivatives of order $2$
\begin{align*}
   \Phi_{(2)}&=\partial_{ss} \psi^1(0)=-\sigma^2 \; , \\
   \Phi_{(11)}&= \frac{1}{6} \left(2 \times \Phi_{(2)}- \sigma^2 \Phi_{()} \right)=-\frac{1}{2} \sigma^2 \; ,
\end{align*}
which implies that the second moment of $\nu_\infty$ is given by $i^{-2} \Phi_{(11)}=\frac{\sigma^2}{2}$.
Likewise, for the partial derivatives of order $4$ we obtain 
\begin{equation}
\label{equ_Phi4}
\begin{aligned}
	\Phi_{(4)}&=\partial_{ssss} \psi^1(0)=6 \sigma^4  \; , \\
	\Phi_{(31)}&= \frac{1}{6} \left(2 \times \Phi_{(4)}- 3  \times \sigma^2 \Phi_{(2)}- 3 \times 2 \sigma^2 \Phi_{(11)} \right)=3\sigma^4 \; , \\
	\Phi_{(22)}&= \frac{1}{6} \left(2 \times \Phi_{(4)}- 4 \times  \sigma^2 \Phi_{(11)}- 2 \times 2 \sigma^2 \Phi_{(2)} \right)=3\sigma^4 \; , \\	
	\Phi_{(211)}&= \frac{1}{12} \left(4 \times \Phi_{(31)}+2 \times \Phi_{(22)}- \sigma^2 (2 \times \Phi_{(11)}+\Phi_{(2)})- 2 \sigma^2 \Phi_{(2)} \right)=\frac{11}{6} \sigma^4 \; , \\		
	\Phi_{(1111)}&= \frac{1}{20} \left(12 \times \Phi_{(211)}- 6 \times \sigma^2 \,  \Phi_{(11)} \right)=\frac{5}{4} \sigma^4 \;,			
\end{aligned}
\end{equation}
which implies that the fourth moment of $\nu_\infty$ is given by $i^{-4} \Phi_{(1111)}=\frac{5}{4} \sigma^4$. 

Note that the partial derivatives listed in \eqref{equ_Phi4} correspond to the ordered partitions of the integer $4$ listed in \eqref{equ_ordering}. This suggests Algorithm~\ref{alg:cap} to compute the $k$-th moment of the particle distribution given by $i^{-k} \Phi_{\one_k}$.

\begin{algorithm}
\caption{Compute the $k$-th moment of $\nu^1$.}
\label{alg:cap}
\begin{algorithmic}
\Require $k>0$ even
\ForAll {$n=2,4,...,k$}
\State {Generate the ordered list of partitions $P(n)=\{(n), (n-1, 1),..., \one_n \}$ }
\State $\Phi_{(n)} \gets \frac{d^n}{ds^n} \psi^1(0)$
\State $\alpha \gets (n)$
\While{$\alpha \neq \one_n$}
\State $\alpha \gets  \text{next partition in } P(n)$
\State $
  \Phi_\al \gets \frac{1}{\#\alpha (\#\alpha+1)} \left(  
\sum_{i \neq j} \Phi_{\al_{i \rightarrow j}} -  \sum_{\substack{\bt\leq \al\\\ |\bt|=2, \\\ \#\bt=2}} \binom{\al}{\bt} \sg^2 \Phi_{\langle{\al-\bt} \rangle} -\sum_{\substack{\bt\leq \al\\\ |\bt|=2, \\\ \#\bt=1}} \binom{\al}{\bt} 2\sg^2 \Phi_{\langle{\al-\bt} \rangle}  \right)$
\EndWhile
\EndFor
\State $M_k \gets i^{-k} \Phi_{\one_k}$
\State \Return $M_k$
\end{algorithmic}
\end{algorithm}

A tree can be used to store the partial derivatives $\Phi_\alpha$. Each node stores the value of a particular $\Phi_\alpha$ and its leaves are the ordered list partial derivatives with respect to $s_k$ where $k=\#\alpha+1$ \Change{and the order of derivatives with respect to $s_k$ is between $1$ and $\alpha_{k-1}$,} i.e. $\Phi_{(\alpha_1, ..., \alpha_{k-1}, 1)}$, ..., $\Phi_{(\alpha_1, ..., \alpha_{k-1}, \alpha_{k-1})}$.

The algorithm allows us to compute the moments of $\nu_\infty$ as follows,
\begin{equation}
\label{equ_moments}
\begin{aligned}
	\int_0^\infty s^2  \, d\nu_\infty(s) &=  \frac{1}{2} \sigma^2  \; , \\
	\int_0^\infty s^4 \, d\nu_\infty(s) &=  \frac{5}{4} \sigma^4  \; , \\
	\int_0^\infty s^6  \, d\nu_\infty(s) &=  \frac{215}{24} \sigma^6 \approx  8.96 \, \sigma^6 \; , \\
	\int_0^\infty s^8  \, d\nu_\infty(s) &=  \frac{102877}{720} \sigma^8 \approx 142.88 \, \sigma^8 \; ,\\
	& \vdots 
\end{aligned}
\end{equation}
which coincide well with the moments of the empirical single particle distribution which we obtained as time and particle average of a long-time simulation run (Figure~\ref{fig:empirical}),
\begin{equation}
\label{empirical_moments}
\begin{aligned}
M_1 & =0  & M_2 &= 0.492 \, \sigma^2\\
M_3 & =-0.00093 \, \sigma^3 & M_4 &= 1.233 \, \sigma^4\\
M_5 & =-0.027 \, \sigma^5 & M_6 &= 9.174\, \sigma^6\\
M_7  & =-1.262 \, \sigma^7 & M_8 &= 154.376 \, \sigma^8 \; , 
\end{aligned}
\end{equation}

Note that the computation of moments of order $k>60$ is somewhat unfeasible on a standard desktop computer due to the increase of computational cost as illustrated by the number of case distinctions in the proof of Lemma~\ref{lem_Phibound}.

\begin{figure}
	\centering
	\includegraphics[width=0.5\linewidth]{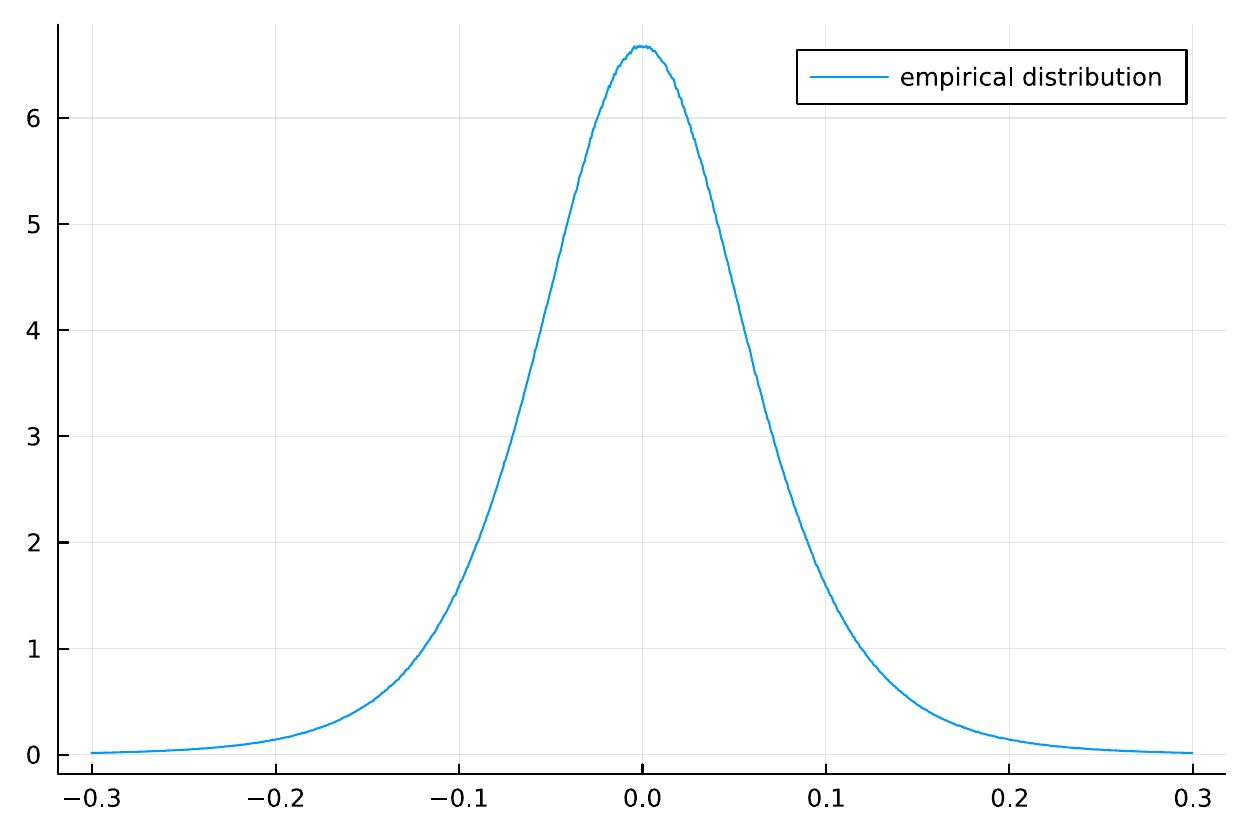}
	\caption{Long-time-average of the single-particle empirical distribution of the renormalised process $\bar X^n_1, ..., \bar X^n_N$ with $N=100$, $\sigma=0.1$ \change{and $\Delta_{n+1} \sim N(0, \sigma^2)$}. The empirical moments are listed in \eqref{empirical_moments}.}
	\label{fig:empirical}
\end{figure}

\section{Conclusion}

In this study, we characterise the asymptotic distribution of particles on the real line undergoing random repositioning into the vicinity of a randomly chosen particle. This stochastic process, defined in \eqref{equ_process1}, can be interpreted as a toy model for actin filament turnover through disassociation and branching~\cite{Gautreau2022,RAZBENAROUSH2017} as well as for opinion formation.

Numerical simulation (figure~\ref{fig:sequenceofframes}) indicates that the distribution of particles is characterised by formation, branching and random turnover and displacement of transient clusters. In this study, however, we factor out the random walk of the centre of mass and aim at characterising the asymptotic distribution (figure~\ref{fig:empirical}) of the centred process, particularly in the limit of large particle numbers, which amounts to finding the shape of the particle cluster. 

Our work investigates the process satisfied by inter-particle distances \eqref{eq:diffs12}, which serves as an equivalent, but simpler formulation of the process satisfied by the centred particle positions. Using that particle positions and their pairwise distances are interchangeable we find that asymptotically, i.e. for large times, the distribution of single inter-particle distances follows a Laplace distribution in agreement with results obtained in other contexts (\cite{KKP-Laplace}).

For the higher-order asymptotic marginal distributions, we establish a recurrence relation between the characteristic functions of the marginal distribution of $k$ inter-particle distances and of $k-1$ inter-particle distances. This fully characterises all time-asymptotic marginal distributions of particle distances in ensembles of size $N$ and opens the way for detailed analysis of the number and spatial distribution of clusters in the case of finite ensemble size.

In the large ensemble limit $N \rightarrow \infty$ this recurrence relation simplifies significantly. It also gives rise to a recurrence relation between partial derivatives of the characteristic functions of inter-particle distance distributions at zero. Using the reconstruction of single particle positions from inter-particle particle distances, we are able to compute all the moments of the asymptotic single particle distribution in the large ensemble limit, illustrated in figure~\ref{fig:empirical}.

Finally, we devise and implement an algorithm to evaluate the recursive scheme for the partial derivatives \eqref{def_Phi} and to compute moments \eqref{equ_moments} of the single-particle distribution exactly.

These results, particularly that the asymptotic single-particle distribution has variance $\sigma^2/2$ and follows a fat-tailed super-Gaussian distribution with kurtosis $5$, are of relevance in applications. For example, in cellular mechanobiology the cluster size in cytoskeleton networks can be assessed experimentally~\cite{Alvarado2013591} in order to inform mechano-chemical models on actin turnover~\cite{Tam2021}.

What has not been considered in this work is the ancestry of single particles \cite{bookDawsonGreven2014} - in the case of our replacement process established through random repositioning events into the vicinity of other particles. 
We remark that the assumption of identical distribution of initial particle positions was used to guarantee that particle distributions are exchangeable. However, we expect that for any given set of initial positions, particle distributions will be exchangeable once all particles share a common ancestor.

Many questions related to the dynamics and structure of the particle clusters are open and should be addressed in future research. This involves further analysis of the distribution of branching events by which transient clusters are shed from existing ones and other characteristics of the random walk of particle clusters, which potentially might be derived through a mean-field procedure \cite{KipnisClaude2014}. A detailed analysis of the number and distribution of transient clusters could be addressed through careful analysis of inter-particle distances and their marginal distributions. 

\Change{Finally, it would also be interesting to further investigate this interacting particle system in the context of non-equilibrium steady states (NESS) \cite{EvansMajumdarPRL2011}} \Change{and stochastic resetting \cite{NagarGupta_2023}.}

\bigskip 
\section*{Acknowledgements} The authors are grateful to Ross McVinish (UQ) for valuable discussions and bibliographic suggestions, \Change{and to the anonymous referees, whose comments and suggestions have allowed us to improve the presentation of this work}.
The authors acknowledge funding from the Australian Research Council (ARC) Discovery Program through grants DP180102956, (awarded to D. B. O.) and DP220102216 (awarded to C. G. T.).

\bibliographystyle{alpha}
\bibliography{aggregates}

\end{document}